\documentclass[11pt]{amsart}
\usepackage{amsmath}
\usepackage{amssymb, verbatim}
\usepackage{tikz}
\usepackage{comment}
\usepackage{color}

\usepackage{graphicx}
\usepackage{caption}
\usepackage{subcaption}
\graphicspath{ {./images/} }

\newtheorem{thm}{Theorem}[section]
\newtheorem{prop}[thm]{Proposition}

\newtheorem{que}{Question}
\newtheorem{defn}[thm]{Definition}
\newtheorem{lem}[thm]{Lemma}
\newtheorem{cor}[thm]{Corollary}
\newtheorem{conj}[thm]{Conjecture}
\newtheorem*{rk}{Remark}

\newcommand{\del}{\partial} 
\newcommand{\dbar}{\overline{\del}}
\newcommand{\ddb}{\sqrt{-1}\del\dbar}

\newcommand{\ddt}{\frac{d}{d t}}

\newcommand{\cP}{\mathcal{P}}
\newcommand{\DDb}{\sqrt{-1}D\overline{D}}

\newcommand{\cblue}[1]{{\textcolor{black}{#1}}} 

\title[Stability and the dHYM equation]{Stability and the deformed Hermitian-Yang-Mills equation}
\author[T. C. Collins]{Tristan C. Collins$^{*}$}
\address{Department of Mathematics\\
 Massachusetts Institute of Technology\\
 77 Massachusetts Ave.\\
 Cambridge\\
 MA 02139\\
 USA}
 \email{tristanc@mit.edu}
\author[Y. Shi]{Yun Shi}
\address{Center of Mathematical Sciences and Applications\\
Harvard University\\
20 Garden St.\\
Cambridge\\
MA 02138\\
USA}
\email{yshi@cmsa.fas.harvard.edu}

\dedicatory{To S.-T. Yau on the occasion of his 70th birthday.}
\thanks{$^{*}$ supported in part by NSF grant DMS-1810924, NSF CAREER grant DMS-1944952, and an Alfred P. Sloan Fellowship }
\date{}							

\begin{document}
\maketitle
\begin{abstract}
We survey some recent progress on the deformed Hermitian-Yang-Mills (dHYM) equation.  We discuss the role of geometric invariant theory (GIT) in approaching the solvability of the dHYM equation, following work of the first author and S.-T. Yau.  We compare the GIT picture with the conjectural picture for dHYM involving Bridgeland stability.  In particular, following Arcara-Miles \cite{AM16}, we show that on the blow-up of $\mathbb{P}^2$ any line bundle admitting a solution of the deformed Hermitian-Yang-Mills equation is Bridgeland stable, but not conversely.  Finally, we survey some recent progress on heat flows associated to the dHYM equation.
\end{abstract}

\section{Introduction}\label{sec: intro}
Let $(X,\omega)$ be a compact Calabi-Yau manifold of real dimension $2n$ with a non-vanishing holomorphic volume form $\Omega \in H^{0}(X, K_{X})$.  A special Lagrangian submanifold of $X$ is a $n$ dimensional submanifold $L\hookrightarrow X$ satisfying the equations
\[
\omega\big|_{L}=0 \qquad {\rm Im}(e^{-\hat{\theta}}\Omega)\big|_{L} =0.    
\]
Special Lagrangian submanifolds were first introduced by Harvey-Lawson \cite{HL1} as a special class of calibrated submanifold.  In particular, a special Lagrangian is automatically volume minimizing in its homology class \cite{HL1}.   The discovery of mirror symmetry further elevated the status of special Lagrangians.  Specializing to the case of $\dim_{\mathbb{R}} X=6$, if one considers type IIA string theory compactified on $X$, special Lagrangian submanifolds (equipped with local systems) are precisely the ``BPS $D$-branes" \cite{BBS}.  Under mirror symmetry these BPS $D$-branes in type IIA (which we will refer to $A$-branes) are related to BPS $D$-branes in the type IIB theory (now called $B$-branes) on the mirror Calabi-Yau manifold \cite{Kont}.  The BPS $B$-branes on $X$ are considerably more mysterious.  At least in the large radius limit (that is, when one replaces $\omega$ with $t\omega$ and lets $t\rightarrow +\infty$ keeping only the leading order terms), the BPS $B$-branes are holomorphic vector bundles with Hermitian-Yang-Mills connections.  Recall the following foundational result
\begin{thm}[Donaldson \cite{Do3}, Uhlenbeck-Yau \cite{UY}]\label{thm: DUY}
Let $(X,\omega)$ be compact K\"ahler and $E\rightarrow X$ an irreducible holomorphic vector bundle.  Then $E$ admits a Hermitian metric $H$ solving the Hermitian-Yang-Mills equation
\[
\Lambda_{\omega}F(H) = c \mathbb{I}_{E}
\]
if and only if $E$ is stable in the sense of Mumford-Takemoto.  That is, for every torsion-free coherent subsheaf $S \subset E$ we have
\[
\mu(E) := \frac{\deg(E)}{{\rm rk}(E)} > \frac{\deg(S)}{{\rm rk}(S)} =: \mu(S).
\]
\end{thm}

Motivated by the correspondence between $A/B$-branes and Theorem~\ref{thm: DUY},  Thomas \cite{Th} and Thomas-Yau \cite{ThY} made the following fundamental conjecture

\begin{conj}[Thomas \cite{Th}, Thomas-Yau, \cite{ThY}]\label{conj: TY}
There is a stability condition on Hamiltonian deformation classes of Lagrangians such that a class $[L]$ is stable if and only if it admits a special Lagrangian representative. 
\end{conj}

Thomas-Yau \cite{ThY} also described a conjectural picture connecting stability, the Lagrangian mean curvature flow and Harder-Narasimhan type filtrations of Lagrangians. The precise notion of stability relevant in Conjecture~\ref{conj: TY} was left open, but Thomas \cite{Th} explained many of the properties that such a stability condition must possess.

At around the same time the analogous picture on the holomorphic side of mirror symmetry was significantly clarified by Mari\~no-Minasian-Moore-Strominger \cite{MMMS}, and Leung-Yau-Zaslow \cite{LYZ} who independently, and using very different points of view, derived the equations of motion for BPS $B$-branes, at least for abelian gauge group; this equation is now known as the {\em deformed Hermitian-Yang-Mills equation}.  Restricting the discussion to the case of zero $B$-field, a holomorphic line bundle $L \rightarrow X$ satisfies the deformed Hermitian-Yang-Mills equation if it admits a metric $h$ such that the curvature $F(h)$ satisfies
\begin{equation}\label{eq: dhymIntro}
\begin{aligned}
{\rm Im}\left(e^{-\sqrt{-1}\hat{\theta}}(\omega + F(h))^{n}\right)=0\\
{\rm Re}\left(e^{-\sqrt{-1}\hat{\theta}}(\omega + F(h))^{n}\right)>0
\end{aligned}
\end{equation} 
for a constant $e^{\sqrt{-1}\hat{\theta}} \in S^1$.  Equation~\eqref{eq: dhymIntro} is a fully nonlinear elliptic equation for the metric $h$. 

A second significant advance came when Douglas \cite{Doug} introduced an algebraic approach to studying BPS $A/B$-branes building on Theorem~\ref{thm: DUY} and ideas from physics. Douglas proposed an algebraic notion of $\Pi$-stability with the same functorial properties that BPS $A/B$-branes should possess.  This opened the door to attempting to study the BPS $A/B$-branes purely algebraically, dispensing with the PDEs describing the equations of motion.  In the context of Theorem~\ref{thm: DUY}, Douglas' proposal amounts to studying Mumford-Takemoto stable bundles without the Hermitian-Yang-Mills equation.   Bridgeland developed this proposal introducing the foundational notion of a Bridgeland stability condition \cite{Br}.  If Douglas' algebraic approach to BPS $A/B$-branes indeed produces objects which satisfy the equations of motion, then we are forced to arrive at the following fundamental conjecture

\begin{conj}\label{conj: newTY}
Let $L \rightarrow X$ (resp. $L \hookrightarrow X$) be a holomorphic line bundle (resp. Lagrangian submanifold).  Then $L$ admits a metric solving the deformed Hermitian-Yang-Mills equation (resp. $L$ can be deformed by Hamiltonian deformations to a special Lagrangian) if and only if $L$ is stable in the sense of Bridgeland as an object in $D^{b}{\rm Coh}(X)$ (resp. $D^{b}{\rm Fuk}(\check{X})$).
\end{conj}

For the time being we will refrain from specifying which particular Bridgeland stability condition on $D^{b}{\rm Coh}(X)$ (resp. $D^{b}{\rm Fuk}(\check{X})$).  However, it should be emphasized that the relevant Bridgeland stability conditions are not known to exist in general, despite a great deal of progress \cite{BMT, AB, Li}.  Thus,  Conjecture~\ref{conj: newTY} should be viewed as two conjectures, the first concerning the existence of Bridgeland stability conditions, and the second concerning the equivalence of stability and the solvability of geometric PDEs. On the symplectic side, Joyce \cite{J} has provided a detailed update to the Thomas-Yau conjecture, connecting the behavior of the Lagrangian mean curvature flow and Bridgeland stability.

The purpose of this article is to survey recent progress towards understanding the existence of solutions to the deformed Hermitian-Yang-Mills equation, with a particular focus on recent work of the first author and S.-T. Yau \cite{CY18}.  We should emphasize that much interesting work has been done on the dHYM equation, associated heat flows, and coupled versions of the dHYM equation \cite{JY, CJY15, HJ1, HY, Ping, Tak, Tak1, ScSt}.  Via mirror symmetry, there are also concrete connections between the dHYM equation and symplectic geometry, particularly in the setting of Landau-Ginzburg models \cite{CY18}.  Due to time and space limitations, we will only be able to touch briefly on some of these themes.

Informally, the perspective taken in \cite{CY18} is the opposite of Douglas' work \cite{Doug}; the first author and Yau take the point of view that the relevant stability condition should appear naturally as obstructing the existence of solutions to the deformed Hermitian-Yang-Mills equation. By developing a GIT (geometric invariant theory) approach to the deformed Hermitian-Yang-Mills equation (building on work of Thomas \cite{Th} and Solomon \cite{Sol} in symplectic geometry), the first author and Yau and found algebro-geometric obstructions to the existence of solutions to the dHYM equation which are at least formally similar to the obstructions which one would expect to find in a conjectural Bridgeland stability condition.  In this article we review this work, as well as some recent progress, and analyze how the stability conditions obtained in \cite{CY18} compare with Bridgeland stability in the case of ${\rm Bl}_{p}\mathbb{P}^2$.  In particular, following work of Arcara-Miles \cite{AM16} we show that a line bundle $L\rightarrow {\rm Bl}_{p}\mathbb{P}^2$ admitting a metric solving the dHYM equation is Bridgeland stable, but not conversely.

The layout of the article is as follows.  In Section~\ref{sec: GIT} we introduce the infinite dimensional GIT framework for the dHYM equation, and discuss the associated variational problem.    In Section~\ref{sec: Alg} we explain how the results in Section~\ref{sec: GIT} can be used to obtain algebraic obstructions to the existence of solutions for the dHYM equation.  We also discuss the role of Chern number inequalities, and the appearance of Bridgeland-type stability structures.  In Section~\ref{sec: Brid} we discuss the work of Arcara-Miles \cite{AM16} and explain the relationship between the stability notions described in Section~\ref{sec: Alg} and Bridgeland stability.  Finally, in Section~\ref{sec: AnAsp} we discuss some of the analytic aspects of the dHYM equation and related heat flows.
\\
\\
{\bf Acknowledgements:} The authors would like to thank A. Jacob for helpful comments on a preliminary draft of this survey, and for explaining to them the results of \cite{JS}.  The authors are also grateful to S.-T. Yau for his interest and encouragement. The first author would like to thank C. Scarpa for helpful comments. The second author would like to thank D. Arcara for answering a question about \cite{ABCH13}.
 
 \section{GIT and the Variational Framework}\label{sec: GIT}
 
 Let $(X,\omega)$ be a compact K\"ahler manifold and fix a class $[\alpha] \in H^{1,1}(X,\mathbb{R})$.  Suppose that the topological constant
 \begin{equation}\label{eq: topConst}
  \hat{z}([\omega],[\alpha]):=\int_{X}(\omega+\sqrt{-1}\alpha)^{n} \in \mathbb{C}^{*}.
 \end{equation}
 Note that in general in can happen that $\hat{z}([\omega],[\alpha])=0$; we will elaborate on this point later.  Assuming $\hat{z}\ne 0$ we can write
 \[
 \hat{z}([\omega],[\alpha]) \in \mathbb{R}_{>0} e^{\sqrt{-1}\hat{\theta}}
 \]
 where $\hat{\theta} \in [0, 2\pi)$ is uniquely determined.  Fix a smooth representative $\alpha$ of the fixed class $[\alpha]$.  The deformed Hermitian-Yang-Mills equation seeks a function $\phi: X\rightarrow \mathbb{R}$ such that $\alpha_{\phi} := \alpha + \ddb \phi$ satisfies
 \[
 {\rm Im}(e^{-\sqrt{-1}\hat{\theta}}(\omega+ \sqrt{-1}\alpha_{\phi})^{n}) =0.
 \]
 To connect this equation with the discussion in the introduction one takes $[\alpha] = -c_{1}(L)$, or more generally $[\alpha] = -c_{1}(L) + [B]$ for some $[B] \in H^{1,1}(X,\mathbb{R})/H^{1,1}(X, \mathbb{Z})$ which corresponds to considering the deformed Hermitian-Yang-Mills equation on the line bundle $L^{-1}$ with respect to the complexified K\"ahler form $\omega+\sqrt{-1}B$; or in other words, `turning on the $B$-field' in the discussion of Section~\ref{sec: intro}.  To get a better feel for this equation it is useful to rewrite the problem.  Fix a point $p\in X$, and choose coordinates $(z_1, \ldots z_n)$ so that at $p$ we have
 \[
 \omega= \sum_{i=1}^{n}\sqrt{-1} dz_i \wedge d\bar{z}_{i}, \qquad \alpha= \sum_{i=1}^{n} \lambda_i \sqrt{-1}dz_i\wedge d\bar{z}_{i}
 \]
 for $\lambda_i \in \mathbb{R}$, $1 \leq i \leq n$.  Define the Lagrangian phase operator to be
 \[
 \Theta_{\omega}(\alpha) := \sum_{i=1}^{n}\arctan(\lambda_i).
 \]
 The deformed Hermitian-Yang-Mills equation is equivalent to
 \[
\Theta_{\omega}(\alpha)= \beta \quad \text{ where } \quad \beta = \hat{\theta}  \mod 2\pi.
 \]
Written in this way it is clear that the deformed Hermitian-Yang-Mills equation is a fully nonlinear elliptic PDE.  

The remainder of this section will describe the infinite dimensional GIT approach to the dHYM equation.  We refer readers unfamiliar with geometric invariant theory to \cite{GIT, ThNo} for a thorough introduction.  The role of infinite dimensional GIT in complex geometric analysis problems dates back at least to Atiyah-Bott \cite{AtBo}, but has been more recently emphasized by work of Donaldson, Hitchin and many others; see for example \cite{Do, Do1, Do4, Do5, Hit, Hit1}.

\subsection{The moment map} The first basic observation, which is inspired by work of Thomas \cite{Th} in symplectic geometry, is that the deformed Hermitian-Yang-Mills equation is the zero of a certain moment map.  Fix a complex line bundle $L \rightarrow X$, and let $h$ be a hermitian metric on $L$.  Let
\[
\mathcal{M} = \{ h \text{ unitary connections } \nabla = d+A \text{ on } L \text{ such that } F_{A}^{0,2}=0\}
\]
That is, $\mathcal{M}$ is the space of $h$-unitary connections on $L$ inducing integrable complex structures on $L$.  Since any integrable complex structure induces a unique $h$-unitary connection on $L$ (the Chern connection), we can view $\mathcal{M}$ as the space of integrable complex structures on $L$.  At any point $A\in \mathcal{M}$ we have
\[
T_{A}\mathcal{M} = \{ B^{0,1} \in C^{\infty}(X, \Lambda^{0,1}T^{*}X) : \dbar B^{0,1}=0\}
\]
An element $B^{0,1} \in T_{A}\mathcal{M}$ gives rise to a local curve in $\mathcal{M}$ by setting $\gamma(t) = A+ t(B^{0,1} + B^{1,0})$ where $B^{1,0}= - \overline{B^{0,1}}$, as dictated by the unitarity requirement.  There is a natural complex structure $\mathcal{J}$ on $T_{A}\mathcal{M}$ defined by
\[
\mathcal{J} B^{1,0} = \sqrt{-1}B^{1,0} \qquad \mathcal{J}B^{0,1} = -\sqrt{-1}B^{0,1}
\]
Define a symplectic form on $\mathcal{M}$ by the following formula
\[
\Omega_{A}(B,C) := \int_{X} B\wedge C \wedge {\rm Re}(e^{-\sqrt{-1}\hat{\theta}}(\omega -F_{A})^{n-1}).
\]
We then get an induced Hermitian inner product by 
\[
\langle B,C \rangle_{A} = \Omega_{A}(B, \mathcal{J}C) = 2\int_{X}{\rm Im}\left(B^{0,1}\wedge\overline{C^{0,1}}\right) \wedge {\rm Re}(e^{-\sqrt{-1}\hat{\theta}}(\omega -F_{A})^{n-1})
\]
In general this inner product is degenerate, but close enough to a solution of the dHYM equation it is a genuine inner product.  The gauge group $\mathcal{G}_{U}$ of unitary transformations of $L$ acts on the right on $\mathcal{M}$ in the usual way and it is easy to see that this action preserves the symplectic form.   The Lie algebra $\mathfrak{g}_{U}$ can be identified with $C^{\infty}(X,\mathbb{R})$ by
\[
C^{\infty}(X,\mathbb{R}) \ni \phi \mapsto g = e^{\sqrt{-1}\phi}.
\]
The dual of the lie algebra $\mathfrak{g}_{U}^{*}$ can be identified with the space of volume forms on $X$ by the non-degenerate pairing $(\phi, \beta) \in \mathfrak{g}_{U}\times\mathfrak{g}_{U}^{*} \mapsto \int_X \phi \beta$.  Now any $\phi \in \mathfrak{g}_{U}$ induces a vector field on $\mathcal{M}$ given by $\sqrt{-1}d\phi = \sqrt{-1}\,\dbar \phi + \sqrt{-1}\del \phi$.  For any $C \in T_{A}\mathcal{M}$ we have
\[
\begin{aligned}
\Omega_{A} (\sqrt{-1}d\phi, C) &= \int_{X} \sqrt{-1}d\phi \wedge C \wedge {\rm Re}(e^{-\sqrt{-1}\hat{\theta}}(\omega -F_{A})^{n-1})\\
&=-\sqrt{-1}\int_{X} \phi dC  \wedge {\rm Re}(e^{-\sqrt{-1}\hat{\theta}}(\omega -F_{A})^{n-1})
\end{aligned}
\]
On the other hand we have $dC = \del C^{0,1} + \dbar C^{1,0}$ since $\dbar C^{0,1}=0$ and $C^{1,0} = -\overline{C^{0,1}}$ by definition.  In particular, $dC$ is purely imaginary and we have
\[
\begin{aligned}
\frac{d}{dt}\big|_{t=0}{\rm Im}\left(e^{-\sqrt{-1}\hat{\theta}}(\omega - F_{A+tC} )^{n}\right)&= -n{\rm Im} \left(e^{-\sqrt{-1}\hat{\theta}}(\omega- F_{A})^{n-1} \wedge dC\right)\\
&=n\sqrt{-1} dC \wedge {\rm Re} \left((e^{-\sqrt{-1}\hat{\theta}}(\omega- F_{A})^{n-1}\right)
\end{aligned}
\]
Thus we see that the moment map for the $\mathcal{G}_{U}$ action is precisely
\[
\mathcal{M} \ni A \mapsto -\frac{1}{n}{\rm Im}\left(e^{-\sqrt{-1}\hat{\theta}}(\omega - F_{A} )^{n}\right) \in \mathfrak{g}_{U}^*.
\]
\subsection{The Kempf-Ness function and GIT} In order to apply the ideas of finite dimensional GIT, one needs a complexification of the group $\mathcal{G}_{U}$, which we denote $\mathcal{G}_{U}^{\mathbb{C}}$, and a Riemannian metric on $\mathcal{G}_{U}^{\mathbb{C}}/\mathcal{G}_{U}$ making it into a non-positively curved symmetric space.  In this setting $\mathcal{G}_{U}^{\mathbb{C}}$ is just the group of complex gauge transformations on $L$, and, using the metric $h$, one can identify the group $\mathcal{G}_{U}^{\mathbb{C}}/\mathcal{G}_{U}$ with the space of hermitian metrics on $L$, or equivalently (modulo scaling), as the space of $(1,1)$-forms  in $c_1(L)$.  To facilitate our discussion of the case when $[\alpha] \in H^{1,1}(X,\mathbb{R})$ we make the following definition, inspired by work of Solomon \cite{Sol}.

\begin{defn}\label{defn: H}
Let $[\alpha] \in H^{1,1}(X,\mathbb{R})$.  Given a function $\phi \in C^{\infty}(X,\mathbb{R})$, we say that a $(1,1)$ form $\alpha_{\phi} := \alpha + \ddb \phi$ in $[\alpha]$ is {\em almost calibrated} if
\[
{\rm Re}(e^{-\sqrt{-1}\hat{\theta}}(\omega + \sqrt{-1}\alpha_{\phi})^n) >0.
\]
We define the space of almost calibrated $(1,1)$ forms to be
\[
\mathcal{H}:= \{ \phi \in C^{\infty}(X,\mathbb{R}) : \alpha_{\phi} \text{ is almost calibrated } \}.
\]
\end{defn}

It can happen that $\mathcal{H}$ is empty.  For example, note that if $\mathcal{H} \ne \emptyset$, then the integral~\eqref{eq: topConst} is non-vanishing since
\[
{\rm Re}(e^{-\sqrt{-1}\hat{\theta}}\hat{z}([\omega],[\alpha]) = \int_{X}{\rm Re}(e^{-\sqrt{-1}\hat{\theta}}(\omega + \sqrt{-1}\alpha_{\phi})^n) >0.
\]
Hence $\hat{z}([\omega],[\alpha])=0$ is a non-trivial intersection theoretic obstruction to $\mathcal{H}\ne \emptyset$.  It would be interesting to give a Nakai-Moishezon type criterion which is equivalent to $\mathcal{H}\ne \emptyset$.  

\begin{que}\label{que: DP}
Are there algebraic conditions on $[\alpha]$ (eg. intersection theoretic conditions) which guarantee that $\mathcal{H}$ is non-empty?
\end{que}

When $\mathcal{H}$ is non-empty, as we shall assume from now on, it is an open subset of $C^{\infty}(X,\mathbb{R})$.  The space of almost calibrated $(1,1)$ forms can be written in terms of the Lagrangian phase operator as
\begin{equation}\label{eq: HdefBranch}
\mathcal{H} = \bigsqcup_{\{\beta \in (-n\frac{\pi}{2}, n \frac{\pi}{2}) : \beta = \hat{\theta} \mod 2\pi \}} \{ \phi \in C^{\infty}(X)~|~ |\Theta_{\omega}(\alpha_{\phi}) - \beta| < \frac{\pi}{2}\}
\end{equation}

An easy argument using the maximum principle shows that either $\mathcal{H}$ is empty, or the disjoint union on the right hand side of~\eqref{eq: HdefBranch} collapses to only one branch \cite{CXY}.  That is, there is a unique $\beta \in (-n\frac{\pi}{2}, n \frac{\pi}{2})$ such that $\beta = \hat{\theta} \mod 2\pi$ and
\begin{equation}\label{eq: liftAng}
\mathcal{H}= \{ \phi \in C^{\infty}(X, \mathbb{R})~|~ |\Theta_{\omega}(\alpha_{\phi}) - \beta| < \frac{\pi}{2}\} \ne \emptyset
\end{equation}

\begin{defn}\label{defn: liftedAngle}
Suppose $\mathcal{H}\ne \emptyset$.  
\begin{enumerate}
\item[(i)] We define the lifted angle, denoted $\hat{\theta}$, to be the uniquely defined {\em lifted phase} $\hat{\theta} \in (-n\frac{\pi}{2}, n \frac{\pi}{2})$ such that~\eqref{eq: liftAng} holds.
\item[(ii)] We say that $[\alpha]$ has {\em hypercritical phase} (with respect to $\omega$) if the lifted phase $\hat{\theta} \in ((n-1)\frac{\pi}{2}, n \frac{\pi}{2})$.
\end{enumerate}
\end{defn}

The space $\mathcal{H}$ is related under mirror symmetry to the space of positive Lagrangians introduced by Solomon \cite{Sol}.  Building off that analogy we can define a Riemannian structure on $\mathcal{H}$; since $\mathcal{H}$ is an open subset of the vector space $C^{\infty}(X,\mathbb{R})$, the tangent space $T_{\phi}\mathcal{H}$ is naturally identified with $C^{\infty}(X,\mathbb{R})$.  Define an inner product by
\[
T_{\phi}\mathcal{H} \ni \psi_1, \psi_2 \longmapsto \langle \psi_1, \psi_2 \rangle_{\phi} := \int_{X} \psi_1\psi_2 {\rm Re}(e^{-\sqrt{-1}\hat{\theta}}(\omega + \sqrt{-1}\alpha_{\phi})^n)
\]
The following result, which is the complex analogue of a result of Solomon \cite{Solomon14}, shows that the ``quotient group" $\mathcal{G}_{U}^{\mathbb{C}}/\mathcal{G}_{U}$ is non-positively curved.

\begin{thm}[Chu-C.-Lee, \cite{CCL}]
The infinite dimensional Riemannian manifold $(\mathcal{H}, \langle \cdot, \cdot \rangle)$ is non-positively curved.
\end{thm}

Associated to the Riemannian structure is a notion of geodesics.  A map $[0,1] \ni s \mapsto \phi(s) \in \mathcal{H}$ is a geodesic if and only if it solves
\begin{equation}\label{eq: geoEq}
\ddot{\phi} + \frac{n\sqrt{-1}\del \dot{\phi} \wedge \dbar\dot{\phi} \wedge {\rm Im}(e^{-\sqrt{-1}\hat{\theta}}(\omega+\sqrt{-1}\alpha_{\phi})^{n-1})}{{\rm Re}(e^{-\sqrt{-1}\hat{\theta}}(\omega+\sqrt{-1}\alpha_{\phi})^{n})}=0
\end{equation}
where $\dot{\phi}= \frac{\del \phi}{\del s}$, and $\ddot{\phi} = \frac{\del^2 \phi}{\del s^2}$.  It turns out \cite{CY18} that the geodesic equation~\eqref{eq: geoEq} can be written as a fully nonlinear degenerate elliptic equation.  Consider the complex manifold $\mathcal{X} := X \times \mathcal{A}$ where $\mathcal{A} := \{ e^{-1} < |t| < 1\} \subset \mathbb{C}$, and let $\pi_{\mathcal{A}} : \mathcal{X} \rightarrow \mathcal{A}$, $\pi_{X}: \mathcal{X}\rightarrow X$ be the projections.  Define
\begin{equation}\label{eq: refMet}
\hat{\omega}_0 := \pi_{X}^*\omega, \qquad \hat{\omega}_{\epsilon} = \pi_{X}^* \omega + \epsilon^{2}\sqrt{-1}dt \wedge d\bar{t}.
\end{equation}
We have
\begin{lem}
Let $\phi_0, \phi_1 \in \mathcal{H}$.  Suppose $\Phi(x,t) = \Phi(x,|t|)$ is an $S^1$ invariant solution of the boundary value problem
\begin{equation}\label{eq: geoBdry}
\begin{aligned}
&{\rm Im}\left(e^{-\sqrt{-1}\hat{\theta}}\left(\hat{\omega}_0 +\sqrt{-1}(\pi_{X}^*\alpha + \DDb \Phi(x,t))\right)^{n+1}\right)=0\\
&\Phi(x,t)\big|_{|t|=1} = \phi_0 \qquad \Phi(x,t)\big|_{|t|=e^{-1}} = \phi_1
\end{aligned}
\end{equation}
where $\DDb$ denotes the complex Hessian on $\mathcal{X}$.  Then $\phi(x,s) := \Phi(x, -\log|t|)$ solves~\eqref{eq: geoEq} and is a geodesic in $\mathcal{H}$.  Conversely, if $\phi(x,s)$ is a geodesic in $\mathcal{H}$, then $\Phi(x,t) = \Phi(x,|t|) = \phi(x,e^{-s})$ is an $S^1$ invariant solution of~\eqref{eq: geoBdry}.
\end{lem}

\begin{rk}\label{rk: noReg}
{\rm Equation~\eqref{eq: geoBdry} is a fully nonlinear, degenerate elliptic equation.  In particular, the existence of smooth solutions is not guaranteed.  In fact, as pointed out in \cite{CCL}, it follows from work of Lempert-Vivas \cite{LV13} and Darvas-Lempert \cite{DL12} that solutions to~\eqref{eq: geoBdry} can have {\em at best} $C^{1,1}$ regularity in general.}
\end{rk}

A real analogue of~\eqref{eq: geoBdry} for domains in $\mathbb{R}^n$ was studied previously by Rubinstein-Solomon \cite{RuSol} and Darvas-Rubinstein \cite{DarRu}.  Let $\Omega \subset \mathbb{R}^n$ be a domain.  Consider the space of Lagrangian graphs $\Omega \ni x \mapsto (x, \nabla f(x)) \in \mathbb{R}^{2n} = \mathbb{C}^n$, where $f: \Omega \rightarrow \mathbb{R}$.  Recall that the phase of a Lagrangian graph is given by
\[
\Theta(D_x^2f) = \sum_{i=1}^{n}\arctan(\lambda_i)
\]  
where $\lambda _i, 1 \leq i \leq n$ are the eigenvalues of $D^{2}f$.  Motivated by Solomon's Riemannian metric on the space of almost calibrated Lagrangians,  Rubinstein-Solomon \cite{RuSol} studied the following PDE problem on $\Omega \times [0,1] \subset \mathbb{R}^{n+1}$; let $(x,t)\in \mathbb{R}^n\times \mathbb{R}$, be coordinates and denote by $\mathbb{I}_n= \sum_i dx_i \otimes dx_i$ the $n\times n$ identity matrix, and $\pi(x,t) = x$ the projection from $\mathbb{R}^{n+1}$ to $\mathbb{R}^n$.  Suppose $f_0, f_1$ are two almost calibrated potentials with
\[
\Theta(D_x^2 f_i) \in (c-\frac{\pi}{2}, c+\frac{\pi}{2}) \quad i=1,2.
\]
Consider the following PDE problem
\begin{equation}\label{eq: RSeq}
\begin{aligned}
&{\rm Im}\left(e^{-\sqrt{-1}\hat{\theta}}\det \left(\pi^*\mathbb{I}_n +\sqrt{-1} D_{x,t}^2 F(x,t)\right)\right)=0\\
&{\rm Re} \left(e^{-\sqrt{-1}\hat{\theta}}\det \left(\mathbb{I}_n +\sqrt{-1} D_{x}^2 F(x,t)\right)\right)>0\\
\end{aligned}
\end{equation}
where $\hat{\theta} = c \mod 2\pi$, with boundary data
\begin{equation}\label{eq: RSbdry}
\begin{aligned}
&F(x,0) =f_0(x) \qquad F(x,1) = f_1(x), \\
&F(x,t)|_{\del \Omega \times [0,1]} = (1-t)f_0(x) + t f_1(x)
\end{aligned}
\end{equation}
To explain the notation, the first line in~\eqref{eq: RSeq} consists of a determinant of a complex valued $(n+1)\times (n+1)$ matrix, while the second line (which only fixes $\hat{\theta}$) is a determinant of an $n\times n$ complex matrix.  Regarding the boundary values, there are several natural candidates for boundary values on the spatial boundary $\del \Omega \times [0,1]$.  While imposing linear boundary data is natural analytically, a more geometric approach would be to impose a kind of Neumann boundary data which makes~\eqref{eq: RSeq} the geodesic equation for Lagrangian graphs.  More precisely, in calculating the geodesic equation for Lagrangian graphs starting from Solomon's metric there is an integration by parts, which in the case of a domain introduces a non-trivial boundary term.  Forcing this boundary term to vanish would yield a a nonlinear Neumann type boundary data on $\del \Omega \times [0,1]$ which is natural geometrically.   

The approach of Rubinstein-Solomon \cite{RuSol} is to recast~\eqref{eq: RSeq} in terms of an elliptic operator they call the lifted space-time Lagrangian angle, which fits into the robust Dirichlet Duality theory of Harvey-Lawson \cite{HL}.  By extending the Harvey-Lawson theory to domains with corners the authors prove
\begin{thm}[Rubinstein-Solomon, \cite{RuSol}]
Suppose $\Omega \subset \mathbb{R}^n$ is bounded and strictly convex.  Then~\eqref{eq: RSeq} admits a continuous viscosity solution in the sense of Harvey-Lawson.  Furthermore, $F(\cdot, t)$ is Lipschitz in $t$, and, if $|c|\in [n\frac{\pi}{2}, (n+1)\frac{\pi}{2})$, then $F$ is Lipschitz in $x$ and $t$.
\end{thm}

In the case $|c| \in [n\frac{\pi}{2}, (n+1)\frac{\pi}{2})$, Darvas-Rubinstein \cite{DarRu} showed that the solution of~\eqref{eq: RSeq} can be obtained from an envelope construction, extending classical work of Kiselman \cite{Kis}.  Ross-Witt Nystr\"om have subsequently proven far reaching generalizations of this result \cite{RWN}.  While the condition $|c| \in [n \frac{\pi}{2}, (n+1)\frac{\pi}{2})$ is natural analytically, in geometric situations only the case $|c| < n\frac{\pi}{2}$ can arise.   Dellatorre \cite{Del} extended the techniques of Rubinstein-Solomon to the case of Riemannian manifolds, while Jacob \cite{Jac} applied similar techniques to establish the existence of continuous viscosity solutions to the geodesic equation~\eqref{eq: geoBdry}.

As we will explain below, in order to have applications to the existence of solutions to the dHYM equation we need to obtain solutions to the geodesic equation~\eqref{eq: geoBdry} with significantly more regularity than can be obtained by general viscosity techniques.  For this reason it is convenient to introduce an elliptic regularization of~\eqref{eq: geoBdry}, which we call the $\epsilon$-geodesic equation.  Fix $\phi_0, \phi_1 \in \mathcal{H}$.  A curve $\phi(s) \in \mathcal{H}$ is an $\epsilon$-geodesic if the associated function $\Phi(x,t) = \phi(x,e^{-s})$ on $\mathcal{X}$ solves the equation
\begin{equation}\label{eq: epsGeo}
\begin{aligned}
&{\rm Im}\left(e^{-\sqrt{-1}\hat{\theta}}\left(\hat{\omega}_\epsilon +\sqrt{-1}(\pi_{X}^*\alpha + \DDb \Phi_{\epsilon}(x,t))\right)^{n+1}\right)=0\\
&\Phi_{\epsilon}(x,t)\big|_{|t|=1} = \phi_0 \qquad \Phi_{\epsilon}(x,t)\big|_{|t|=e^{-1}} = \phi_1
\end{aligned}
\end{equation}
where $\hat{\omega}_{\epsilon}$ is defined in~\eqref{eq: refMet}.  

Let us return to our discussion of the infinite dimensional GIT problem.  Define a complex $1$-form on $\mathcal{H}$ by the following formula
\[
T_{\phi}\mathcal{H} \ni \psi \mapsto dCY_{\mathbb{C}}(\psi) = \int_{X} \psi (\omega+\sqrt{-1}\alpha_{\phi})^{n}.
\]
After fixing a base point, this $1$-form integrates to a well-defined functional $CY_{\mathbb{C}}: \mathcal{H} \rightarrow \mathbb{C}$ \cite{CY18}.  More explicitly, assume the reference form $\alpha$ is almost calibrated, so that we can take $0 \in \mathcal{H}$ as the reference point.  Then we have the formula \cite{CY18}
\[
CY_{\mathbb{C}}(\phi) = \frac{1}{n+1}\sum_{j=0}^{n} \int_{X}\phi(\omega+\sqrt{-1}\alpha_{\phi})^j\wedge (\omega+\sqrt{-1}\alpha)^{n-j}.
\]  

The crucial observation is the following

\begin{lem}[C.- Yau, \cite{CY18}]\label{lem: func}
With notation as above, we have
\begin{enumerate}
\item[(i)] The function 
\[
\mathcal{C}(\phi) := {\rm Re}\left(e^{-\sqrt{-1}\hat{\theta}} CY_{\mathbb{C}}(\phi)\right)
\]
is affine along (putative) smooth geodesics in $\mathcal{H}$.
\item[(ii)] The function 
\[
\mathcal{J} := - {\rm Im}\left(e^{-\sqrt{-1}\hat{\theta}} CY_{\mathbb{C}}(\phi)\right)
\]
is the Kempf-Ness functional for the GIT problem associated to the dHYM equation.  In particular, we have
\[
d\mathcal{J}(\psi) = -\int_{X} \psi {\rm Im}\left(e^{-\sqrt{-1}\hat{\theta}}(\omega+\sqrt{-1}\alpha_{\phi})^{n}\right)
\]
and $\mathcal{J}$ is convex along (putative) smooth geodesics in the space $\mathcal{H}$.
\item[(iii)] Suppose $[\alpha]$ has hypercritical phase, so that $\hat{\theta} \in ((n-1)\frac{\pi}{2}, n\frac{\pi}{2})$.  Define the functional
\[
Z(\phi): = e^{-\sqrt{-1}\frac{n \pi}{2}} CY_{\mathbb{C}}(\phi).
\]
Then ${\rm Re}(Z), {\rm Im}(Z)$ are concave along (putative) smooth geodesics in $\mathcal{H}$.
\end{enumerate}
\end{lem}

The key point is that the function $\mathcal{J}$ defined in Lemma~\ref{lem: func} (ii) is convex along geodesics, and has critical points at solutions of the dHYM equation.  Thus, we are reduced to determining whether the convex function $\mathcal{J}$ has a critical point in $\mathcal{H}$.  At least formally, this is determined by the behavior of the slope of $\mathcal{J}$ on $\del \mathcal{H}$; see Theorem~\ref{thm: CAT0} and the following discussion for an interpretation of $\del \mathcal{H}$.  More concretely, suppose that there is a solution $\phi_{dHYM} \in \mathcal{H}$ of the dHYM equation with hypercritical phase.  Let $\phi(s), s\in[0, \infty)$ be a geodesic ray in $\mathcal{H}$,  with $\phi(0)= \phi_{dHYM}$.  Since $\phi(0)$ is a critical point of $\mathcal{J}$ we have $\frac{d}{ds}|_{s=0}\mathcal{J}(\phi(s)) =0$.  On the other hand, since $\mathcal{J}(\phi(s))$ is convex we must have
\[
\lim_{s\rightarrow \infty} \frac{d}{ds} \mathcal{J}(\phi(s)) = \lim_{s\rightarrow \infty} \frac{\mathcal{J}(\phi(s))-\mathcal{J}(\phi(0))}{s} \geq 0
\]
Thus, if there is a geodesic ray with $\phi(s)$ with $\lim_{s\rightarrow \infty} \frac{d}{ds} \mathcal{J}(\phi(s)) <0$ then no solution of the dHYM equation exists.  Conversely, if the limit slope is positive {\em for all} geodesic rays, then one expects $\mathcal{J}$ to have a critical point in $\mathcal{H}$.  This is the essence of the Kempf-Ness theorem and the Hilbert-Mumford criterion in finite dimensional GIT \cite{GIT, ThNo}.

We now come to the first central issue in utilizing the GIT framework to study the dHYM equation; in order for the approach to work one needs geodesics $\phi(s)$ with enough regularity that the functional $\mathcal{J}(\phi(s))$ is {\em defined} and can be shown to be convex as a function of $s$.  This is a subtle matter.  As discussed above, the Harvey-Lawson viscosity theory (\'a la Rubstein-Solomon \cite{RuSol} and Jacob \cite{Jac}) produces only continuous geodesics, along which $\mathcal{J}$ is not obviously defined.  Furthermore, according to Remark~\ref{rk: noReg} the geodesic equation in $\mathcal{H}$ does not admit even $C^{2}$ solutions, while the calculation proving Lemma~\ref{lem: func} $(ii)$ assumes the existence of geodesics with at least $C^{3}$ regularity; see \cite{CY18}. Nevertheless we have the following result

 \begin{thm}[C.-Yau, \cite{CY18}, Chu-C.-Lee \cite{CCL}]\label{thm: geoThm}
 Suppose $[\alpha]$ has hypercritical phase in the sense of Definition~\ref{defn: liftedAngle} (ii).  Then, for any $\phi_0, \phi_1 \in \mathcal{H}$ there exists a unique $C^{1,1}$ solution of the geodesic equation~\eqref{eq: geoBdry}.  Furthermore, the functional $CY_{\mathbb{C}}$ is defined along this curve and conclusions of Lemma~\ref{lem: func} hold. 
 \end{thm}

Let us briefly explain the outline of the proof of Theorem~\ref{thm: geoThm}.  In \cite{CY18} the first author and Yau proved the existence of $C^{1,\alpha}$ geodesics (for any $\alpha \in [0,1)$) connecting points in $\mathcal{H}$ under the hypercritical phase assumption.  This is done by proving, for all $\epsilon >0$, the existence of a unique, smooth $\epsilon$-geodesic $\Phi_{\epsilon}$ solving~\eqref{eq: epsGeo}, with {\em uniform} estimates
\[
\|\Phi_{\epsilon}\|_{L^{\infty}(\mathcal{X})} + \|\nabla \Phi_{\epsilon}\|_{L^{\infty}(\mathcal{X}, \hat{\omega}_{1})} +  \|\DDb \Phi_{\epsilon}\|_{L^{\infty}(\mathcal{X}, \hat{\omega}_{1})}  \leq C
\]
where $\hat{\omega}_1$ is defined in~\eqref{eq: refMet}.  Furthermore, it is shown that the conclusions of Lemma~\ref{lem: func} hold for these $\epsilon$-geodesics.  Theorem~\ref{thm: geoThm} is obtained by passing to the limit as $\epsilon \rightarrow 0$, making use of some results in pluripotenial theory.   It was furthermore shown that the resulting $C^{1,\alpha}$ functions could be interpreted as solutions of~\eqref{eq: geoBdry} in the pluripotential sense, or as viscosity solutions of the associated Dirichlet problem in the viscosity sense of Rubinstein-Solomon \cite{RuSol}, and Harvey-Lawson \cite{HL}.  In \cite{CCL} the authors improved the regularity to full $C^{1,1}$ regularity, and obtained the following result

\begin{thm}[Chu-C.-Lee, \cite{CCL}]
Suppose $[\alpha] \in H^{1,1}(X,\mathbb{R})$ has hypercritical phase, and let $\phi_0, \phi_1 \in \mathcal{H}$, $\phi_0\ne \phi_1$.  The infinite dimensional Riemannian manifold $(\mathcal{H}, \langle \cdot, \cdot \rangle)$ has a well-defined metric structure, with distances given by
\[
d(\phi_0, \phi_1) = \lim_{\epsilon \rightarrow 0} \int_{0}^1 \sqrt{\langle \dot{\phi}_{\epsilon}(s), \dot{\phi}_{\epsilon}(s)\rangle} ds
\]
where $\phi_{\epsilon}(s)$ is the unique, smooth $\epsilon$-geodesic in $\mathcal{H}$ connecting $\phi_0, \phi_1$.
\end{thm}

Note that the geodesics joining points in $\mathcal{H}$ do not themselves lie in $\mathcal{H}$, since $\mathcal{H}$ consists of {\em smooth} potentials.  For this reason it is convenient to introduce the completion of the metric space $(\mathcal{H}, \langle \cdot, \cdot \rangle)$, which we denote by $(\widetilde{\mathcal{H}}, \tilde{d})$.  It turns out that most of the properties of $(\mathcal{H}, \langle \cdot, \cdot \rangle )$ carry over to this larger space.

\begin{thm}[Chu-C.-Lee, \cite{CCL}]\label{thm: CAT0}
The metric space $(\widetilde{\mathcal{H}}, \tilde{d})$ is a $CAT(0)$ length space.
\end{thm}

As a consequence of Theorem~\ref{thm: CAT0}, and the general properties of $CAT(0)$ length spaces, the space $(\widetilde{\mathcal{H}}, \tilde{d})$ has an intrinsically defined boundary.  The upshot of this result and Theorem~\ref{thm: geoThm} is that, at least when $[\alpha]$ has hypercritical phase, we can employ the powerful philosophy of finite dimensional GIT to study the dHYM equation.  In the next section we will explain how Theorem~\ref{thm: geoThm} can be used to find algebro-geometric obstructions to the existence of solutions to the dHYM equation.

 \section{Algebraic Obstructions, Chern Number Inequalities and Stability}\label{sec: Alg}
 
 Suppose $[\alpha]\in H^{1,1}(X,\mathbb{R})$ has hypercritical phase.  The follow construction is motivated by work of Ross-Thomas \cite{RT}; fix a filtration of coherent ideal sheaves
 \[
 \mathfrak{I}_0 \subset \mathfrak{I}_1 \subset \cdots \mathfrak{I}_{r-1} \subset \mathfrak{I}_{r} := \mathcal{O}_{X}
 \]
 Consider the flag ideal $\mathfrak{J} \subset \mathcal{O}_{X}\otimes \mathbb{C}[t]$ given by
 \begin{equation}\label{eq: flag}
 \mathfrak{J} = \mathfrak{I}_0 + \mathfrak{I}_1\cdot t + \cdots+\mathfrak{I}_{r-1} t^{r-1} + (t^r)
 \end{equation}
To this data we can associated a locally defined, $S^1$-invariant plurisubharmonic function $\psi(x,t) = \psi(x,|t|)$ on $X\times \Delta$ ($\Delta:= \{ t\in \mathbb{C} : |t|\leq 1\}$ in the following way; let $(f_{\ell, k})_{k=1}^{N_{\ell}}$ be local holomorphic functions generating $\mathfrak{I}_{\ell}$.  Then we set
\[
\psi(x,t) = \frac{1}{2\pi}\log\left(\sum_{\ell=0}^{r} |t|^{2\ell} \sum_{k=1}^{N_{\ell}} |f_{\ell,k}|^2\right).
\]
The function $\psi$ is smooth away from $\{t=0\}$ and singular on ${\rm supp}(\mathfrak{J})$.  Demailly-P\u{a}un \cite{DP} show that these local models can be glued together to obtain a globally defined $S^1$-invariant function, which we also denote by $\psi(x,t)$, satisfying
\[
\DDb \psi \geq -A(\pi^*\omega + \sqrt{-1}dt\wedge d\bar{t})
\]
on $X\times \Delta$.  From this estimate it follows that for each $\phi_0 \in \mathcal{H}$ there is a $\delta >0$ such that
\[
\Phi(x,t) = \phi_0(x) + \delta \psi(x,t) 
\]
defines a curve in $\mathcal{H}$ which approaches $\del \mathcal{H}$ as $t\rightarrow 0$.  While this curve is not a geodesic, we can use it as a target for geodesics, as illustrated in Figure~\ref{fig: scheme}

\begin{figure}[h]
\begin{center}
\begin{tikzpicture}[scale=1.3]
\draw[color=gray, thick](0,0) rectangle (6,4);
\node [left] at (6,3) {$\mathcal{H}$};
\node [left] at (1,2) {$\phi_{0}$};
\draw[fill](1,2) circle [radius=0.05];
\draw [dashed, thick, color=gray] plot [smooth] coordinates {(1,2) (2,3.5) (3,.5) (5,2) (6,1.5)};
\draw[thick] (1,2)--(3.65,.8);
\draw[thick] (1,2)--(4.55, 1.68);
\end{tikzpicture}
\caption{The curve $\Phi(x,t)$ is denoted by the dashed gray line, while the black lines indicate geodesics joining $\phi_0$ to points on $\Phi(x,t)$ as $t\rightarrow 0$.}\label{fig: scheme}
\end{center}
\end{figure}
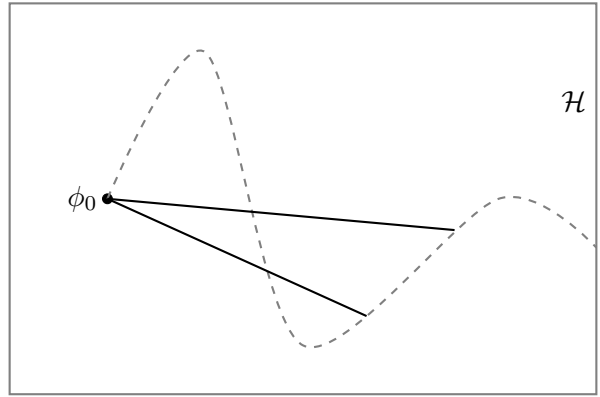

If $\phi_0$ is a solution of the dHYM equation, then by the convexity of $\mathcal{J}$ along geodesics we have
\[
0 \leq \frac{\mathcal{J}(\Phi(x,|t|))- \mathcal{J}(\phi_0) }{-\log|t|}
\]
for all $|t|>0$.  We have the following calculation
\begin{prop}[C.-Yau, \cite{CY18}]
Let $\mu : \widetilde{\mathcal{X}}\rightarrow X\times \Delta$ be a log-resolution of singularities of the ideal sheaf $\mathfrak{J}$ so that $\mu^{-1}\mathfrak{J}= \mathcal{O}_{\widetilde{\mathcal{X}}}(-E)$ for a simple normal crossings divisor $E$.  Then we have
\[
\lim_{t\rightarrow 0} \frac{\mathcal{J}(\Phi(x,|t|))- \mathcal{J}(\phi_0) }{-\log|t|} = \frac{\delta}{\pi} E. {\rm Im}\left(e^{-\sqrt{-1}\hat{\theta}}(\mu^{*}\omega + \sqrt{-1}(\mu^*\alpha -\delta E))^n\right)
\]
\end{prop}

As a consequence, we obtain the following corollary, which yields algebraic obstructions to the existence of solutions of the deformed Hermitian-Yang-Mills equation.

\begin{cor}
Suppose $[\alpha] \in H^{1,1}(X,\mathbb{R})$ has hypercritical phase, and admits a solution of the dHYM equation.  Then for every flag ideal $\mathfrak{J}$ as in~\eqref{eq: flag}, and any log-resolution  $\mu : \widetilde{\mathcal{X}}\rightarrow X\times \Delta$ of $\mathfrak{J}$  so that $\mu^{-1}\mathfrak{J}= \mathcal{O}_{\widetilde{\mathcal{X}}}(-E)$ for a simple normal crossings divisor $E$ we must have
\begin{equation}\label{eq: stabIn}
\frac{\delta}{\pi} E. {\rm Im}\left(e^{-\sqrt{-1}\hat{\theta}}(\mu^{*}\omega + \sqrt{-1}(\mu^*\alpha -\delta E))^n\right) \geq 0
\end{equation}
for all $\delta>0$ sufficiently small.
\end{cor}

The inequality in~\eqref{eq: stabIn} can be improved to a strict inequality via a perturbation argument, provided the flag ideal is not of the form $\mathfrak{J} = (t^r)$ for some $r\geq 0$ \cite{CY18}. In fact one gets even more information by applying the same reasoning to the functionals $Z, \mathcal{C}$ introduced in Lemma~\ref{lem: func}.  Let us introduce the following quantities;
for any flag ideal $\mathfrak{J}$ and log-resolution $\mu: \tilde{X}\rightarrow X\times \Delta$ with $\mu^{-1}\mathfrak{J} = \mathcal{O}_{\widetilde{\mathcal{X}}}(-E)$ as above, let
\[
Z_{\delta E} = -\delta E. e^{-\sqrt{-1}\omega + (\alpha-\delta E)}, \qquad Z_{X}([\alpha]) = -[X]. e^{-\sqrt{-1}\omega +\alpha}.
\]
Note that if $\alpha= c_1(L)$ for some holomorphic line bundle then
\[
Z_{\delta E} = -\delta E. e^{-\sqrt{-1}\omega}ch(L-\delta E), \qquad Z_{X} ([\alpha])= -[X]. e^{-\sqrt{-1}\omega}ch(L).
\]
The algebraic obstructions of \cite{CY18} can be summarized as
\begin{prop}[C.-Yau, \cite{CY18}]\label{prop: algObst}
Suppose $[\alpha] \in H^{1,1}(X,\mathbb{R})$ has hypercritical phase (in particular $\mathcal{H} \ne \emptyset$).  Then
\begin{enumerate}
\item[(i)] For every flag ideal $\mathcal{J}$ and every log-resolution as above, we have
\[
{\rm Re}\left(\frac{Z_{\delta E}}{Z_{X}([\alpha])}\right) \geq 0
\]
and
\[
{\rm Im}\left(Z_{\delta E}\right) >0
\]
for all $\delta >0$ sufficiently small.
\item[(ii)] If in addition $[\alpha]$ admits a solution of the dHYM equation then
\[
{\rm Im}\left(\frac{Z_{\delta E}}{Z_{X}([\alpha])}\right) \geq 0
\]
for all $\delta >0$ sufficiently small.
\end{enumerate}
Furthermore, if equality holds in either $(i)$ or $(ii)$ then we must have $\mathfrak{J} = (t^{r})$ for some $r\geq 0$.
\end{prop}

Said another way, if either of the inequalities in Proposition~\ref{prop: algObst} (i) fail, then $[\alpha]$ cannot have hypercritical phase, and if the inequality in (ii)  fails, then $[\alpha]$ cannot admit a solution of the dHYM equation.  To get a better feel for these obstructions it is useful to consider the simplest possible flag ideal; namely, when $r=1$ and $\mathfrak{I}_0 = I_{V}$ is the ideal sheaf of a proper irreducible subvariety $V\subset X$, with $\dim_{\mathbb{C}} V=p$.  In this case a standard calculation in intersection theory \cite[Lemma 7.5]{CY18} shows that
\[
Z_{\delta E} = -\delta^{n-p} \binom{n}{n-p}e^{\sqrt{-1}n\frac{\pi}{2}}\int_{V}e^{-\sqrt{-1}(\omega+\sqrt{-1}\alpha)} + O(\delta^{n+1-p})
\]
The leading order terms in $\delta$ as $\delta \rightarrow 0$ is proportional to
\[
Z_{V}([\alpha]) := -\int_{V}e^{-\sqrt{-1}(\omega+\sqrt{-1}\alpha)},
\]
which yields

\begin{cor}\label{cor: obstructions}
Suppose $[\alpha] \in H^{1,1}(X,\mathbb{R})$ has hypercritical phase (in particular $\mathcal{H} \ne \emptyset$), and $V\subset X$ is a proper, irreducible subvariety of dimension $p <n$ .  Then
\begin{enumerate}
\item[(i)]
\[
{\rm Re}\left(\frac{Z_{V}([\alpha])}{Z_{X}([\alpha])}\right)>0 
\]
and
\[
{\rm Im}\left(Z_{V}([\alpha])\right) >0.
\]
\item[(ii)] If in addition $[\alpha]$ admits a solution of the dHYM equation then
\[
{\rm Im}\left(\frac{Z_{V}([\alpha])}{Z_{X}([\alpha])}\right) > 0
\]
for all $\delta >0$ sufficiently small.
\end{enumerate}
\end{cor}

\subsection{The case of K\"ahler surfaces}
Let us consider the situation when $\dim X=2$.  It is a consequence of the Hodge index theorem that, for any $[\alpha]\in H^{1,1}(X,\mathbb{R})$ we have
\[
Z_{X}([\alpha] )= -\int_{X}e^{-\sqrt{-1}(\omega + \sqrt{-1}\alpha)} = \int_{X}(\omega+\sqrt{-1}\alpha)^2 \in \mathbb{C}^*
\]
Furthermore, it is not hard to see that the lifted angle, when it is defined, is given by
\[
\hat{\theta} = {\rm Arg}(Z_{X}([\alpha]))
\]
where ${\rm Arg}$ denotes the principal value of the argument.  We consider the case when ${\rm Im}(Z_{X}([\alpha])) >0$ (otherwise we can replace $[\alpha]$ with $-[\alpha]$).  The deformed Hermitian-Yang-Mills equation can be written as
\[
-\sin(\hat{\theta})(\omega^2-\alpha^2) + 2\cos(\hat{\theta})\omega\wedge\alpha =0.
\]
Since $\sin(\hat{\theta}) >0$ by assumption, we can rewrite this equation as
\[
(\cot(\hat{\theta})\omega + \alpha)^2=(1+\cot^2(\hat{\theta})) \omega^2.
\]
which is the Monge-Amp\`ere equation for the cohomology class $[\cot(\hat{\theta})\omega+\alpha] \in H^{1,1}(X,\mathbb{R})$.  By Yau's solution of the Calabi conjecture \cite{Y}, this equation has a solution if and only if $\pm [\cot(\hat{\theta})\omega+\alpha]$ is a K\"ahler class.  Since ${\rm Im}(Z_{X}([\alpha])) >0$ it is not hard to show that only $[\cot(\hat{\theta})\omega+\alpha]$ can be K\"ahler.   We have the following lemma

\begin{lem}[C.-Jacob-Yau, \cite{CJY15}]\label{lem: conjDim2}
Let $(X,\omega)$ be a K\"ahler surface, and $[\alpha] \in H^{1,1}(X,\mathbb{R})$ such that ${\rm Im}(Z_{X}([\alpha])) >0$.  Then there exists a solution of the dHYM equation if and only if for every curve $C\subset X$ we have $Z_{C}([\alpha]) = -\int_{C}e^{-\sqrt{-1}(\omega+\sqrt{-1}\alpha)}$ satisfies
\[
{\rm Im}\left(\frac{Z_{C}([\alpha])}{Z_{X}([\alpha])}\right) >0
\]
\end{lem}
\begin{proof}
We only sketch the proof.  Rewriting the condition ${\rm Im}\left(\frac{Z_{C}([\alpha])}{Z_{X}([\alpha])}\right) >0$ yields
\[
\int_{C} \cot(\hat{\theta})\omega + \alpha >0.
\]
It then follows from the Demailly-P\u{a}un theorem \cite{DP} that $[\cot(\hat{\phi})\omega+\alpha]$ is K\"ahler; see \cite{CJY15}.
\end{proof}

Lemma~\ref{lem: conjDim2} shows that the obstructions from Corollary~\ref{cor: obstructions} are both necessary and sufficient in dimension $2$, even without the hypercritical phase assumption.

\subsection{An algebraic framework for stability}

Our goal now is to assemble the obstructions from Corollary~\ref{cor: obstructions} into a {\em purely algebraic} framework which is conjecturally equivalent to the existence of solutions to the dHYM equation.  In order to do this we first need to discuss the lifted angle.  Recall that the lifted angle $\hat{\theta}$ is defined by assuming that $\mathcal{H}\ne \emptyset$, see Definition~\ref{defn: liftedAngle}.  From now on, to make the exposition more clear, we will call this that {\em analytic lifted angle} and denote it by $\hat{\theta}_{an}([\alpha])$.  If one had a positive answer to Question~\ref{que: DP} then the condition that $\mathcal{H} \ne \emptyset$ would be algebraic.  However, determining the lifted angle would {\em still} be essentially analytic, since it depends on understanding the phase of an almost calibrated $(1,1)$ form.  In order to circumvent this difficulty, we give a purely algebraic approach to determining the lifted angle.  This was originally discussed in \cite{CXY}. 

Given a class $[\alpha] \in H^{1,1}(X,\mathbb{R})$ and an irreducible subvariety $V\subset X$ (with $V=X$ allowed) we consider the curve
\begin{equation}\label{eq: sliceMot}
Z_{V, [\alpha]}(t) = -\int_{V}e^{-\sqrt{-1}(t\omega+ \sqrt{-1}\alpha)} = -\frac{(-\sqrt{-1})^{\dim V}}{(\dim V)!} \int_{X}(t\omega + \sqrt{-1}\alpha)^{\dim V}
\end{equation}
where $t$ runs from $+\infty$ to $1$.  We also denote
\[
Z_{V}([\alpha]) := Z_{V,[\alpha]}(1) =  -\int_{V}e^{-\sqrt{-1}(t\omega+ \sqrt{-1}\alpha)}. 
\]
We now make the following definition.

\begin{defn}\label{defn: algLiftAngle}
With the above notation
\begin{enumerate}
\item[(i)] We define the algebraic lifted angle $\hat{\theta}_{V}([\alpha])$ to be the winding angle of the curve $Z_{V, [\alpha]}(t)$ as $t$ runs from $+\infty$ to $1$, provided $Z_{V, [\alpha]}(t) \in \mathbb{C}^*$ for all $t\in[1,\infty]$.
\item[(ii)] We define the slicing angle by
\[
\varphi_{V}([\alpha]) = \hat{\theta}_{V}([\alpha]) -   \frac{\pi}{2}(\dim V-2)
\]
and note that
\[
Z_{V,[\alpha]}(1) \in \mathbb{R}_{>0}e^{\sqrt{-1}\varphi_{V}([\alpha])}.
\]
\end{enumerate}
\end{defn}

The definition of the slicing angle is motivated by~\eqref{eq: sliceMot}. 

\begin{rk}
{\rm There is no obvious ``canonical" path through the space $H^{1,1}(X,\mathbb{C})$ for computing the algebraic lifted angle.  Indeed, the path $[t\omega+\sqrt{-1}\alpha]$ is just one of many possible choices;  different choices of path can yield different values for the algebraic lifted angle.  However, a natural requirement is that the choice of path computes the analytic lifted angle whenever the latter is defined.  Or perhaps less stringently, that the algebraic lifted angle computes the analytic lifted angle when a solution of the dHYM equation exists.  See Lemma~\ref{lem: CNI} below for such a result in dimension $3$.}
\end{rk}

Let us recap the three different angles defined on $X$;
\begin{itemize}
\item The analytic lifted angle $\hat{\theta}_{an}([\alpha])$, defined when $\mathcal{H}\ne \emptyset$, as in Definition~\ref{defn: liftedAngle}.
\item The algebraic lifted angle, $\hat{\theta}_{X}([\alpha])$ defined in Definition~\ref{defn: algLiftAngle} (i).
\item The slicing angle, $\varphi_{X}([\alpha])$ defined in Definition~\ref{defn: algLiftAngle} (ii), which agrees with the lifted angle up to a normalization factor.
\end{itemize}

The only obstacle to defining the algebraically lifted angle is the possibility that $Z_{V, [\alpha]}(T) =0$ for some $T\in [1,\infty]$, since if this occurs the winding angle is no longer well-defined.   When $\dim V=1$, it is easy to see that ${\rm Im}(Z_{V,[\alpha]}(t)) >0$ for all $t$, and hence this cannot occur.  When $\dim V=2$ then $Z_{V,[\alpha]}(t) \in \mathbb{C}^*$ by the Hodge index theorem.  However, when $\dim V >2$ then it happens that $Z_{V,[\alpha]}(t)$ can pass through the origin; examples occur on $Bl_p\mathbb{P}^3$.  Thus, each of the analytic and algebraic lifted angles suffer some deficiencies in general.  On the other hand, we have the following result

\begin{lem}[C.-Xie-Yau, \cite{CXY}]\label{lem: CNI}
Suppose $X$ has dimension $3$ and $[\alpha]$ admits a solution of the dHYM equation with analytic lifted angle $\hat{\theta}_{an}([\alpha]) \in (\frac{\pi}{2}, 3\frac{\pi}{2})$.  Then
\begin{equation}\label{eq: CNI}
\left(\int_{X}\omega^3\right)\left(\int_{X} \frac{\alpha^3}{3!}\right) < 3 \left(\int_{X} \frac{\alpha^2\wedge \omega}{2!}\right)\left(\int_{X} \alpha \wedge \omega^2\right).
\end{equation}
In particular, $Z_{X,[\alpha]}(t) \in \mathbb{C}^*$ for all $t\in [1,\infty]$, and $\hat{\theta}_{X}([\alpha]) = \hat{\theta}_{an}([\alpha])$.
\end{lem}

Note that~\eqref{eq: CNI} is precisely the line bundle case of the inequality conjectured by Bayer-Macri-Toda \cite[Conjecture 3.2.7]{BMT} to hold for tilt-stable objects in their construction of Bridgeland stability conditions.  Schmidt \cite{Sch} has shown that such an inequality cannot hold for tilt-stable objects in general by constructing counterexamples on $Bl_p\mathbb{P}^3$.  It is interesting that the inequality {\em does hold} if instead of til-stability we impose the solvability of the dHYM equation.

Combining Lemma~\ref{lem: CNI} with the above discussion we get

\begin{cor}\label{cor: algLift}
Suppose $X$ has dimension $3$ and $[\alpha]$ admits a solution of the dHYM equation with phase $\hat{\theta} \in (\frac{\pi}{2}, 3\frac{\pi}{2})$.  Then the algebraic lifted angle $\hat{\theta}_{V}([\alpha])$ is well defined for any $V\subset X$, including $V=X$, and we have
\[
\hat{\theta}([\alpha])_{an} =  \hat{\theta}_{X}([\alpha]).
\]
\end{cor}

Corollary~\ref{cor: algLift} allows us to recast our obstructions in a purely algebraic way in dimension $3$.  In order to facilitate our comparison with Bridgeland stability conditions in the next section, let us assume from now on that $[\alpha]=c_1(L)$.  Assuming $L$ admits a solution of the dHYM equation with right hand side $\hat{\theta}_{an}(L) \in (\pi, 3\frac{\pi}{2})$ one can easily check that
\[
Z_{X}(L) = -\int_{X}e^{-\sqrt{-1}\omega}ch(L) \in \{ z \in \mathbb{C} : {\rm Im}(z)>0, \, {\rm Re}(z) <0 \}
\]
By Corollary~\ref{cor: algLift} the algebraic lifted angle $\hat{\theta}_{X}(L)$ is equal to the $\hat{\theta}_{an}(L)$.  Next, by Corollary~\ref{cor: obstructions}, (i), for any $V\subset X$ irreducible analytic subvariety the complex numbers 
\[
Z_{V}(L) = -\int_{V} e^{-\sqrt{-1}\omega}ch(L) \in  \{ z \in \mathbb{C} : {\rm Im}(z)>0 \}
\]
Furthermore, by Corollary~\ref{cor: obstructions}, (ii) we have
\[
{\rm Im}\left( \frac{Z_{V}(L)}{Z_{X}(L)}\right) >0.
\]
We illustrate this situation in Figure~\ref{fig: stabPic}.

\begin{figure}
\begin{center}
\begin{tikzpicture}[scale=1.5]
\draw[lightgray!50, fill=lightgray!50](3,1) -- (0,2.5) -- (0,1);
\draw[thick] (3,3)--(3,0);
\draw[thick] (0,1)--(6,1);
\draw[dashed](4,0) arc [radius =1.8, start angle=-20, end angle = 15];
\draw (4.065,1) arc [radius =1.8, start angle=10, end angle = 133];
\draw[thick, dashed](3,1)--(0,2.5);
\draw[fill](1.04,1.98) circle [radius=0.05];
\node [above] at (1,2.3) {$Z_{X}(L)$};
\node [above] at (4,2) {$Z_{X}(t)$};
\end{tikzpicture}
\caption{The path $Z_{X}(L)(t)$, and its endpoint $Z_{X}(L)$.  If $L$ admits a solution of the dHYM equation with $\hat{\theta} \in (\pi, 3\frac{\pi}{2})$, then $Z_{V}(L)$ must lie in the gray region for every irreducible analytic set $V\subset X$.} \label{fig: stabPic}
\end{center}
\end{figure}
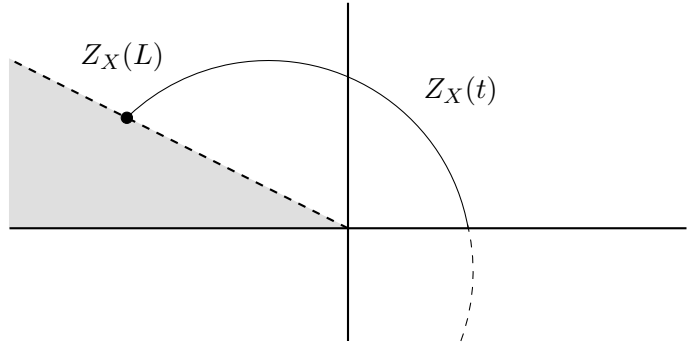

We are now going to determine the slicing angles for each $V\subset X$.  By assumption, we have $\hat{\theta}_{X}(L) \in(\pi, \frac{3\pi}{2})$, and so $\varphi_{X}(L)\in (\frac{\pi}{2}, \pi)$.  If $V$ has dimension $1$, then
\[
Z_{V, L}(t) = -\int_{V}c_1(L) + \sqrt{-1}t\int_{V}\omega.
\]
Since $Z_{V}(L)$ must lie in the shaded region in Figure~\ref{fig: stabPic}, we see that the lifted angle $\hat{\theta}_{V}(L) \in (\varphi_{X}(L)-\frac{\pi}{2},\frac{\pi}{2})$, and so
\[
\varphi_{V}(L) \in(\varphi_{X}(L), \pi).
\]
If $V$ has dimension $2$ then we have
\[
2Z_{V, L}(t) = \int_{V}t^2\omega^2-c_{1}(L)^2 + \sqrt{-1}2t\int_{V}c_{1}(L)\wedge \omega
\]
In this case ${\rm Im}(Z_{V}(L)) >0$, and so $Z_{V, L}(t)$ must lie in $\mathbb{H}$ for all $t\in[1,+\infty)$.  It follows that the lifted angle must satisfy $\hat{\theta}_{V}(L) \in (\varphi_{X}(L), \pi)$.  Since ${\rm dim}V =2$ we get
\[
\varphi_{V}(L) \in (\varphi_{X}(L), \pi).
\]
We summarize this in the following proposition,
\begin{prop}\label{prop: BridgStabInt}
Suppose that $(X,\omega)$ is a K\"ahler 3-fold and $L\rightarrow X$ is a holomorphic line bundle.  If $L$ admits a solution of dHYM with lifted angle $\hat{\theta}\in(\pi, \frac{3\pi}{2})$.  Then
\begin{itemize}
\item[(i)] The Chern number inequality~\eqref{eq: CNI} holds, and so the lifted angle is well-defined.
\item[(ii)] $Z_{X}(L) \in \{z \in \mathbb{C} : {\rm Im}(z)>0\}$, and the slicing angle $\phi_{X}(L) \in (\frac{\pi}{2}, \pi)$.
\item[(iii)]  For every irreducible analytic subset $V\subset X$, $Z_{V}(L) \in \{z \in \mathbb{C} : {\rm Im}(z)>0\}$.
\item[(iv)] For every irreducible analytic subset $V\subset X$, the slicing angle $\phi_{V}(L)$ satisfies
\begin{equation}\label{eq: BridStabIneq}
\varphi_{V}(L) > \varphi_{X}(L).
\end{equation}
\end{itemize}
\end{prop}

It is natural to propose 
\begin{conj}[C.-Yau, \cite{CY18}]\label{conj: CY}
The converse of Proposition~\ref{prop: BridgStabInt} holds.
\end{conj}

Note that the ``small radius limit" of this conjecture for ample line bundles over toric varieties was proven by the first author and Sz\'ekelyhidi \cite{CoSz}, and without the toric assumption in \cite{GC}; see \cite{CXY} for a discussion.  One can formulate analogous conjectures in higher dimensions.  However, in dimension $4$ and higher one needs to impose further Chern number inequalities on $X$, as well as Chern number inequalities on $V\subset X$ to ensure that the slicing angle $\varphi_{V}(L)$ is well-defined; see \cite{CY18} for a discussion, and \cite{GC} for some progress.

There is one situation in which we have an essentially complete solution to Conjecture~\ref{conj: CY}.  Let $X=Bl_{p}\mathbb{P}^n$ be the blow up of the projective space at a point, and let $E$ be the exceptional divisor of the blow-up, and let $H$ denote the proper transform of the hyperplane.  Then $X\backslash (H\cup E) = \mathbb{C}^n$, and following Calabi \cite{Cal} one can look for rotationally invariant $(1,1)$ forms on $\mathbb{C}^n$ which extend over $H\cup E$ to globally defined forms on $X$.  In this case the dHYM equation reduces to an ODE and one can study the existence of solutions using essentially algebraic techniques.  We have the following theorem

\begin{thm}[Jacob-Sheu, \cite{JS}]
Let $X = Bl_p\mathbb{P}^n$, and let $\omega$ be a K\"ahler metric on $X$, and let $[\alpha]\in H^{1,1}(X,\mathbb{R})$ be any class.  Then $[\alpha]$ admits a solution of the dHYM equation if and only if
\[
Z_{X}([\alpha]) \in \mathbb{C}^*
\]
and, for any $V\subset X$ we have
\[
{\rm Im}\left(\frac{Z_{V}([\alpha])}{Z_{X}([\alpha])}\right) >0.
\]
\end{thm}

Note that this result {\em does not} assume that $[\alpha]$ has hypercritical phase. There are several other features of this result which seem to be special to the case of $Bl_{p}\mathbb{P}^n$.  First, in the hypercritical phase case the stability type inequalities ${\rm Im}\left(\frac{Z_{V}([\alpha])}{Z_{X}([\alpha])}\right) >0$ imply that ${\rm Im}(Z_{V}) >0$ for all $V\subset X$; this seems to be non-trivial.  Secondly, the authors bypass the Chern number inequalities needed to define the algebraic lifted phase by giving a method for computing the lifted phase that makes explicit use of the symmetry at hand.  It would be interesting to know whether this could be generalized beyond the case of $Bl_{p}\mathbb{P}^n$.

 \section{Relationship to Bridgeland Stability}\label{sec: Brid}
 
 We would like to compare the algebraic obstructions discussed in the previous section and those of Bridgeland stability. We now recall the definition of a Bridgeland stability condition, focusing specifically on the case of interest for the $B$-model of mirror symmetry, so that the triangulated category is $D^{b}Coh(X)$.
\begin{defn}\label{defn: slicing}
A {\em slicing} $\cP$ of $D^{b}Coh(X)$ is a collection of subcategories $\cP(\varphi) \subset D^{b}Coh(X)$ for all $\varphi \in \mathbb{R}$ such that
\begin{enumerate}
\item $\cP(\varphi)[1] = \cP(\varphi+1)$ where $[1]$ denotes the ``shift" functor,
\item if $\varphi_1 > \varphi_2$ and $A\in \cP(\varphi_1)$, $B \in \cP(\varphi_2)$, then ${\rm Hom}(A,B) =0$,
\item every $E\in D^{b}Coh(X)$ admits a Harder-Narasimhan filtration by objects in $\cP(\phi_i)$ for some $1 \leq i \leq m$.
\end{enumerate}
\end{defn}

We refer to \cite{Br}, or Proposition~\ref{prop: BrStab} below, for a precise definition of the Harder-Narasimhan property.  A Bridgeland stability condition on $D^{b}Coh(X)$ consists of a slicing together with a {\em central charge}.  For BPS $D$-branes in the B-model, the relevant central charge was first proposed by Douglas (see, for example, \cite{Doug, DFR, BMT, AB}).  We take
\[
D^{b}Coh(X) \ni E \longmapsto Z_{D}(E):= -\int_{X}e^{-\sqrt{-1}\omega}ch(E).
\]
Many authors also consider the central charge
\[
D^{b}Coh(X) \ni E \longmapsto Z'_{D}(E):= -\int_{X}e^{-\sqrt{-1}\omega}ch(E)\sqrt{Td(X)}.
\]
This latter central charge seems to be unrelated to the dHYM equation in general, so we will not consider it.
\begin{defn}\label{defn: BrStab}
A Bridgeland stability condition on $D^{b}Coh(X)$ with central charge $Z_{D}$ is a slicing $\cP$ satisfying the following properties
\begin{enumerate}
\item For any non-zero $E\in \cP(\varphi)$ we have
\[
Z_{D}(E) \in \mathbb{R}_{>0} e^{\sqrt{-1}\varphi},
\]
\item
\[
C := \inf \left\{ \frac{|Z_{D}(E)|}{\|ch(E)\|} : 0 \ne E \in \cP(\varphi), \varphi \in \mathbb{R} \right\} >0
\]
where $\| \cdot \|$ is any norm on the finite dimensional vector space $H^{even}(X, \mathbb{R})$.
\end{enumerate}
\end{defn}

Given a Bridgeland stability condition the {\em heart} is defined to be $\mathcal{A} := \cP((0,\pi])$, and this is itself an abelian category.  An object $A \in \mathcal{A}$ is semistable (resp. stable) if, for every surjection $A\twoheadrightarrow B$, $B\in\mathcal{A}$ we have
\[
 \varphi(A) \leq (\text { resp.} <)\,\,\varphi(B).  
\]

Comparing this definition with the obstructions discussed in Section~\ref{sec: Alg} it is seems that, at least aesthetically, the algebraic structures which predict the existence or non-existence of solutions to dHYM are closely related to Bridgeland stability.  For example, in dimension $3$, if $V$ is an irreducible analytic subvariety and $\mathcal{O}_{V}$ is the skyscraper sheaf supported on $V$, then the ideal dictionary would be
\[
\begin{aligned}
\text{ Proposition~\ref{prop: BridgStabInt} (i)-(iii) } &\Longleftrightarrow L, L\otimes \mathcal{O}_{V} \in \mathcal{A}\\
\text{ Proposition~\ref{prop: BridgStabInt} (iv) } &\Longleftrightarrow L \text{ is not destabilized by } L\twoheadrightarrow L\otimes \mathcal{O}_{V}
\end{aligned}
\]
Note however that that $Z_{V}(L)\ne Z_{D}(L\otimes \mathcal{O}_{V})$ in general.  Solutions of the dHYM equation have two more important similarities with Bridgeland stable objects.  First, by a result of Jacob-Yau \cite{JY}, line bundles admitting solutions of dHYM have property (2) of Definition~\ref{defn: BrStab}.  Secondly, the following lemma was proved in \cite{CY18}; it should be compared with Definition~\ref{defn: slicing}, (2).

\begin{lem}[C.-Yau, \cite{CY18}]
Suppose $L_1, L_2$ are two line bundles on $(X,\omega)$ admitting solutions of the deformed Hermitian-Yang-Mills equation with
\[ 
0<\varphi_{X}(L_2)<\varphi_{X}(L_1) < \pi.
\]
Then ${\rm Hom}(L,M) = 0$.
\end{lem}

The remainder of this section is devoted to understanding the correspondence between solutions of the deformed Hermitian-Yang-Mills equation and Bridgeland stable objects in a particular example; the case of $Bl_{p}\mathbb{P}^2$.    All the results presented in the rest of this section were obtained in \cite{AM16} and some of the discussion applies more generally to projective surfaces of Picard rank $2$.  We first recall some basic facts about Bridgeland stability condition on surfaces. The following proposition is proved in \cite{Br}, and it is a very useful description when one wants to construct a Bridgeland stability condition on a smooth variety.
\begin{prop}[Bridgeland, \cite{Br}]\label{prop: BrStab}
	\label{pro1}
	A Bridgeland stability condition on $D^bCoh(X)$ is equivalent to the following data: the heart $\mathcal{A}$ of a bounded t-structure on $D^bCoh(X)$, and a central charge $Z: K(\mathcal{A})\rightarrow \mathbb{C}$ such that for every nonzero object $E\in \mathcal{A}$, one has (i) $Z(E)\in |Z(E)|e^{\sqrt{-1}\pi\phi}$ for $\phi\in (0, 1]$, (ii) $E$ has a finite filtration
	\begin{equation*}
	0=E_0\subset E_1\subset...\subset E_{n-1}\subset E_n=E
	\end{equation*} 
	such that $HN_i(E)=E_i/E_{i-1}$'s are semistable objects in $\mathcal{A}$ with decreasing phase $\phi$. Furthermore, the central charge satisfies Definition \ref{defn: BrStab} (2).
\end{prop}
This filtration is called the Harder-Narasimhan filtration, and the $HN_i$'s are called the HN factors of $E$. Similar to the case of slope/Gieseker stability condition, the existence of such a filtration implies that any object in the abelian category $\mathcal{A}$ can be built up by extensions of semistable objects. For the rest of this note, we will always use the above notations for Harder-Narasimhan filtration and HN factors. 

If a stability condition is given by the above data, we denote the stability condition by 
\begin{equation*}
\sigma=(\mathcal{A}, Z).
\end{equation*} 

Usually one would like the central charge to factor through a finite dimensional vector space $K_{num}=K(\mathcal{A})/\chi(\_, \_)$, where the Euler form $\chi$ is defined by
$$
\chi(E, F)=\sum_i(-1)^idim(Hom(E, F[i])).
$$

For surfaces, a subset of the set of Bridgeland stability conditions which satisfy the above requirements has been constructed, see \cite{AB, BM11, Bri08}. We review this subset of Bridgeland stability conditions here.

Let $\nu$ be an ample class in $NS_\mathbb{R}(X)$. Let $B$ be any class in $NS_\mathbb{R}(X)$. We denote the $B$-twisted Chern character by $ch^B(E)=e^{-B}\cdot ch(E)$.
Consider the central charge 
\begin{equation*}
Z_{\nu, B}(E)=-\int_Xe^{-\sqrt{-1}\nu}\cdot ch^B(E).
\end{equation*}
For any coherent sheaf $E$ on $X$, we denote the $B$-twisted slope of $E$ by 
\begin{equation*}
\mu_\nu^B(E)=\frac{ch_1^B(E)\cdot\nu}{ch_0^B(E)\cdot \nu^2},
\end{equation*}
and the slope with no $B$-twist by $\mu_\nu(E)$. 

Let $Coh^B(X)$ be the abelian category in $D^bCoh(X)$ obtained by tilting the abelian category $Coh(X)$ at the following torsion pair:

\begin{center}
$\mathcal{T}^B_\nu$=$\{E\in Coh(X)|$ any semistable factor $F$ of $E$ satisfies $\mu_\nu^B(F)>0\}$, 
\end{center}
\begin{center}
$\mathcal{F}^B_\nu$=$\{E\in Coh(X)|$ any semistable factor $F$ of $E$ satisfies $\mu_\nu^B(F)\leq 0\}$, 
\end{center}

Equivalently, objects in $Coh^B(X)$ are given by
\begin{equation*}
\{E\in D^bCoh(X)| H^0(E)\in \mathcal{T}^B_\nu, H^{-1}(E)\in \mathcal{F}^B_\nu, H^i(E)=0, i\neq -1, 0\}.
\end{equation*}
The following theorem is proved by many authors in the literature (see, e.g. \cite{ABCH13, BM11, Bri08})
\begin{thm}
	\label{stability}
	The pair $\sigma_{\nu, B}=(Coh^B(X), Z_{\nu, B})$ defines a Bridgeland stability condition on $X$.
\end{thm}
To simplify expression, we use the following notation:
\begin{equation*}
\begin{split}
\rho_{\nu, B}(E)=&-\frac{Re(Z_{\nu, B}(E))}{Im(Z_{\nu, B}(E))}\\
=&\frac{ch_2^B(E)-\frac{1}{2}\nu^2ch_0^B(E)}{\nu\cdot ch_1^B(E)}.
\end{split}
\end{equation*}
For any two objects $E, F\in Coh^B(X)$, the phase of $\sigma_{\nu, B}(E)$ is greater than the phase of $\sigma_{\nu, B}(F)$ if and only if $\rho_{\nu, B}(E)>\rho_{\nu, B}(F)$.  Now we consider the case when $X=Bl_p(\mathbb{P}^2)$, and $\omega$ an ample class in $NS_\mathbb{R}(X)$.  Without loss of generality, we assume that $\omega\cdot\omega=1$. Let $E$ be the exceptional divisor on $X$. Let $H$ be another element in $NS_\mathbb{R}(X)$ such that $H\cdot\omega=0$, $H\cdot H=-1$ \cblue{and $E\cdot H>0$}. We are interested in which line bundles are $\sigma_{\omega, 0}$ stable. Following the idea of \cite{ABCH13}, we consider a real slice of the stability manifold which contains $\sigma_{\omega, 0}$, and further analyze the geometry of potential walls in this real slice to obtain useful restrictions on destabilizing objects. 

Consider a subset of the stability conditions constructed in Theorem \ref{stability} by taking $\nu=x\omega$, and $B=y\omega+zH$ for $x, y, z\in \mathbb{R}, x>0$. We denote a stability condition in this slice by $\sigma_{x, y, z}$. Then $\sigma_{\omega, 0}$ is denoted by
\begin{equation*}
\sigma_{1, 0, 0}=(Coh^0(X), Z_{1, 0, 0}).
\end{equation*}

Consider an arbitrary line bundle $L$ on $Bl_p\mathbb{P}^2$.

\begin{defn}
	A potential wall associated to $L$ is a subset in $\mathbb{R}^{>0}\times\mathbb{R}^2$ consisting of points $(x, y, z)$ such that 
	\begin{equation*}
	ReZ_{x, y, z}(E)ImZ_{x, y, z}(L)=ReZ_{x, y, z}(L)ImZ_{x, y, z}(E), 
	\end{equation*}
	for some object $E$ of $D^bCoh(X)$.
\end{defn}
We denote such a wall by $W(E, L)$. Note that $E$ may not actually destabilize $L$ since $E$ may not be a subobject of $L$ in $Coh^B$.

Using the formula of central charge, we see that the wall is defined by 
\begin{equation}
\label{wall}
(\frac{1}{2}x^2\omega^2ch_0^B(E)-ch_2^B(E))\cdot(x\omega ch_1^B(L))=(\frac{1}{2}x^2\omega^2ch_0^B(L)-ch_2^B(L))\cdot(x\omega ch_1^B(E)).
\end{equation}

Denote the Chern character of $E$ by $(r, c, d)$, where $c=E_\omega \omega+E_H H$.
Similarly, we denote the Chern character of $L$ by $(1, l, e)$, where $l=l_\omega \omega+l_H H$. 
Then the wall $W(E, L)$ defined by equation \ref{wall} is a quadratic surface in $\mathbb{R}^3$:
\begin{equation}
\label{walle}
\begin{split}
\frac{1}{2}x^2(rl_\omega-E_\omega)+\frac{1}{2}y^2(rl_\omega-E_\omega)+\frac{1}{2}z^2(rl_\omega-E_\omega)-yz(rl_H-E_H)\\
+y(d-er)-z(E_Hl_\omega-l_HE_\omega)-(dl_\omega-eE_\omega)=0
\end{split}
\end{equation}

We denote the intersection of $W(E, L)$ and any plane $z=z_0$ by $W(E, L)_{z_0}$. From the equation, we see that $W(E, L)_{z_0}$'s are semicircles if nonempty. We denote the region inside each semicircle by $W(E, L)_{z_0}^<$, and the region outside by $W(E, L)_{z_0}^>$.  Recall the following theorem by Maciocia on the structure of the potential walls:
\begin{thm}[Maciocia, \cite{Mac14}]
	\label{nest}
	Let Y be a smooth projective surface over $\mathbb{C}$, and $F$ a $\mu_\omega$-semistable torsion free sheaf on $Y$. Then the potential walls $W(E, F)_z$ are all semicircles on the upper half plane, furthermore they are nested \cblue{at either side of $y=\mu_\omega(F)$} for any $z\in\mathbb{R}$.
\end{thm}

Here nested means that given two potential walls, one is in the interior of the other wall or vice visa. 
In particular, if two walls $W(E_1, F)_z$ and $W(E_2, F)_z$ intersect for some $z$, then $F$ cannot be $\mu_\omega$-semistable. 

Note that there is a special wall which is an asymptotic wall of all others. This is a plane defined by $y=l_\omega$. Its intersection with any plane $z=z_0$ can be thought of a semicircle with infinite radius. 

\begin{figure}[h]
	\centering
	\includegraphics[scale=0.6]{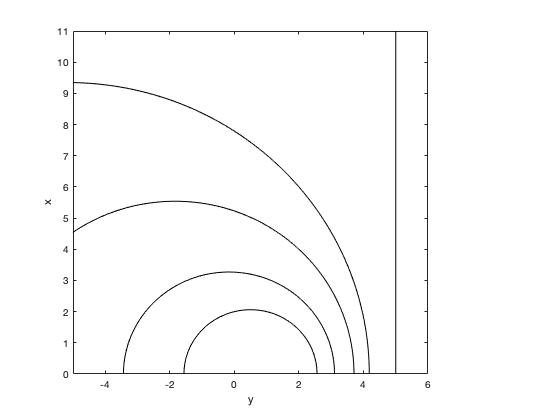}
	\caption{Example of potential walls restricted to a plane $z=z_0$}
\end{figure}

The following proposition is the main idea used in proving  \cite[Proposition 6.2]{ABCH13}. It gives an important tool to break possibly destabilizing objects into sheaves of smaller rank, and thereby reduce considerations to destabilizing objects of rank $1$. This is the strategy we are going to follow.

\begin{prop}[Arcara-Bertram-Coskun-Huizenga, \cite{ABCH13}]
	\label{ABCH}
	Assume that $E$ destabilizes $L$ for a stability condition $\sigma_{x_0, y_0, z_0}$. If the plane $y=\mu_\omega(HN_n(E))$ intersects $W(E, L)_{z_0}$, then $E_{n-1}$ is also a destabilizing object of $L$ at $\sigma_{x_0, y_0, z_0}$. 
\end{prop}

\begin{proof}
	Since $E\in Coh^{y_0}(Y)$, we have $\mu_\omega(HN_n(E))>y_0$. In the plane $z=z_0$, we can approach the ray 
	\begin{equation*}
	y=\mu_\omega(HN_n(E))
	\end{equation*}
	from the left along $W(E, L)_{z_0}$. As $y\to \mu_\omega(HN_n(E))^-$, we know that 
	\begin{equation*}
	\omega\cdot ch_1^B(HN_n(E))\to 0^+.
	\end{equation*}	
	Then we have 
	\begin{equation*}
	\rho_{x, y, z}(HN_n(E))\to -\infty.
	\end{equation*}
	Hence as $y\to \mu_\omega(HN_n(E))^-$, the following is true: 
	\begin{equation*}
	\rho_{x, y, z}(E_{n-1})>\rho_{x, y, z}(E)=\rho_{x, y, z}(L).
	\end{equation*}
	Since $\mu_\omega(HN_n(E))<\mu_\omega(E)<\mu_\omega(E_{n-1})$, the point $(x, y, z_0)\in W(E, L)_{z_0}$ as $y\to\mu_\omega(HN_n(E))^-$ is inside the region $W(E_{n-1}, L)_{z_0}^<$. \cblue{This implies that $W(E_{n-1}, L)_{z_0}\subset W(E, L)_{z_0}^>$}, hence $E_{n-1}$ also destabilizes $L$ at $(x_0, y_0, z_0)$. 
\end{proof}

Now we explore the geometry of $W(E, L)$'s. We first observe that \cblue{for any $E\in D^bCoh(X)$, either} $W(E, L)$ is tangent to the plane $\{y=l_\omega\}$ at the point $(0, l_\omega, l_H)$, or $(0, l_\omega, l_H)$ is a singular point of $W(E, L)$.

This observation implies that the only possible types of quadratic surfaces defined by equation~\eqref{walle} are ellipsoids, elliptic paraboloids and elliptic hyperboloids of two sheets (and its degenerate case, an elliptic cone). Assume further that there is an object $E$ destabilizes $L$ at the stability condition $\sigma_{1, 0, 0}$. Namely there is an exact sequence in $Coh^0(X)$:
\begin{equation}
\label{destabilize}
0\to E\xrightarrow{f} L\xrightarrow{g} F\to 0,
\end{equation} 
and $\rho_{1, 0, 0}(E)>\rho_{1, 0, 0}(L)$. Note that by the long exact sequence on cohomology, $E$ is a coherent sheaf on $X$, and $F$ is quasi-isomorphic to a 2-term complex in $D^bCoh(X)$. 

The point $(1, 0, 0)$ lies in the region $W(E, L)_0^<\cap \{y<HN_n(E)\}$. Then the wall $W(E, L)$ can only be of the following types: 

\begin{itemize}
	\item[{\bf Type 1}] An ellipsoid tangent to the plane $y=l_\omega$ on the left.
	
	\item[{\bf Type 2}] An elliptic paraboloid tangent to the plane $y=l_\omega$ on the left.
	
	\item[{\bf Type 3}] An elliptic hyperboloid of two sheets tangent to the plane $y=l_\omega$ on the left (including the degenerate case: elliptic cone).
	
	\item[{\bf Type 4}] An elliptic hyperboloid of two sheets tangent to the plane $y=l_\omega$ on the right.
\end{itemize}
These situations are illustrated below

\begin{figure}[h]
	\centering
	\begin{minipage}{.5\textwidth}
		\includegraphics[scale=0.35]{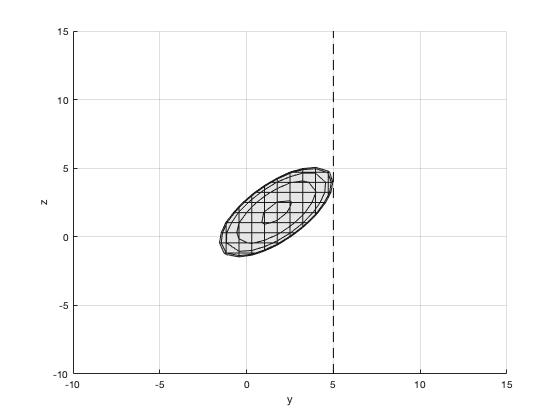}
		\captionof{figure}{Potential wall of type 1}
	\end{minipage}%
	\begin{minipage}{.5\textwidth}
		\centering
		\includegraphics[scale=0.35]{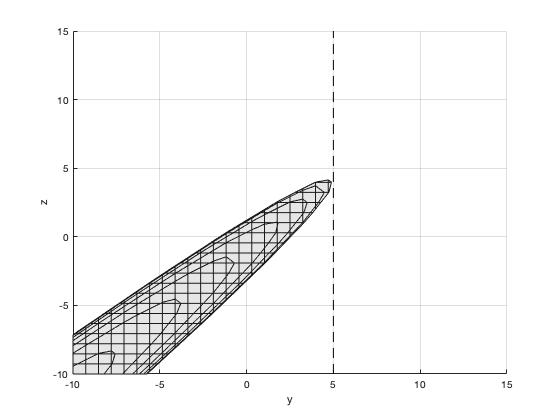}
		\captionof{figure}{Potential wall of type 2}
	\end{minipage}%
\end{figure}

\begin{figure}[h]
	\begin{minipage}{.5\textwidth}
		\centering
		\includegraphics[scale=0.35]{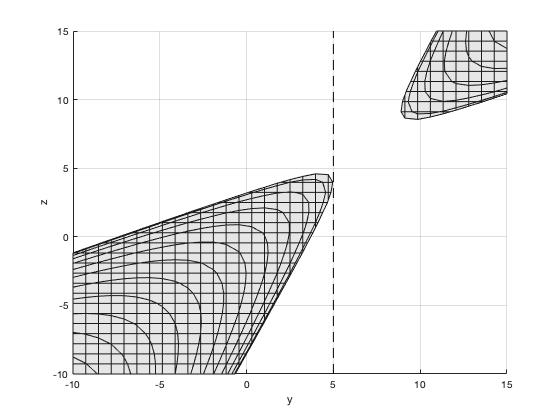}
		\captionof{figure}{Potential wall of type 3}
	\end{minipage}%
	\begin{minipage}{.5\textwidth}
		\centering
		\includegraphics[scale=0.35]{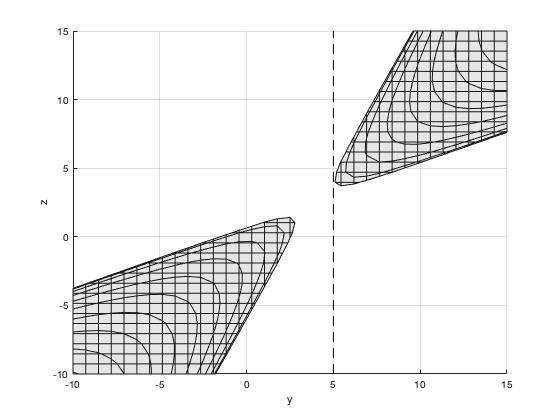}
		\captionof{figure}{Potential wall of type 4}
	\end{minipage}
\end{figure}

\newpage
Using Theorem \ref{ABCH}, we first deduce that
\cblue{
\begin{prop} [Arcara-Miles, \cite{AM16}]
	\label{type123}
	If $E$ destabilizes $L$ at $\sigma_{1, 0, 0}$ and $W(E, L)$ is of type 1, 2 or 3, then $W(E, L)_{0}$ is inside $W(E_i, L)_{0}$ for some $i$ with $W(E_i, L)$ of type 4.
\end{prop}
\begin{proof}
	We have 
	\begin{equation*}
	0<\mu_\omega(HN_n(E))\leq\mu_\omega(E)<l_\omega.
	\end{equation*} 
	Since $W(E, L)$ is of type 1, 2 or 3, the wall $W(E, L)$ intersects the plane $y=\mu_\omega(E)$. By Theorem \ref{nest}, $E$ is not $\mu_\omega$-semistable. 
	WLOG, we assume that the axis of $W(E, L)$ is positive. We denote the range of $z$ such that $W(E, L)_z\cap \{y=\mu_\omega(HN_i(E))\}\neq \emptyset$ by $[a_i, b_i]$. By Proposition \ref{ABCH}, the wall $W(E_{n-1}, L)_z\subset W(E, L)_z^>$ for $z\in [a_n, b_n]$. If $W(E_{n-1}, L)_0\subset W(E, L)_0^<$, then $W(E_{n-1}, L)$ is also of type 1, 2 or 3. Repeat the argument for $E_i$'s inductively. If $W(E_i, L)_0\subset W(E, L)_0^<$ for all $i$, then $W(E_1, L)$ intersects $y=\mu_\omega(E_1)$, contradicts with $E_1$ is $\mu_\omega$ semistable. Hence $W(E, L)_0\subset W(E_i, L)_0^<$ for some $i$. If $W(E_i, L)$ is also of type 1, 2 or 3, we replace $E$ by $E_i$ and find a wall $W(E_j, L)_0\subset W(E_i, L)_0^>$ for some $j<i$. This implies that the wall $W(E_k, L)$ s.t.$W(E_k, L)_0$ is most outside is a wall of type 4. 
\end{proof}
}
	We first consider possibly destabilizing objects of rank $1$. In this case, we have
	\begin{equation*}
	E\simeq L\otimes I
	\end{equation*}
	for some ideal sheaf $I$. Let $Y$ be the subscheme defined by $I$.
	Let $Q$ be the maximal torsion subsheaf of $\mathcal{O}_Y$. 
	Then we have the following exact sequence of $\mathcal{O}_X$ modules:
	\begin{equation*}
	0\to Q\to \mathcal{O}_Y\to \mathcal{O}_C\to 0,
	\end{equation*}
	where $\mathcal{O}_C$ is a pure sheaf of dimension $1$. We denote $E$ by \cblue{$L(-C)_n$}, where $n$ is the length of $Q$. Now the exact sequence~\eqref{destabilize}
	is actually an exact sequence in $Coh(X)$.  The following Proposition is proved in \cite{AM16}.  
	\begin{prop} [Arcara-Miles, \cite{AM16}]
		\label{rk1}
		Let $C$ be a curve on $X$. If $C^2\geq 0$, then $L(-C)_n$ does not destabilize $L$ at $\sigma_{x_0, y_0, z_0}$ for any $(x_0, y_0, z_0)\in \mathbb{R}^{>0}\times \mathbb{R}^2$. 
	\end{prop}
	\begin{proof}
		We write $C=(C\cdot\omega)\omega-(C\cdot H)H$. Let $\Delta$ be the discriminant of the quadratic equation $W(L(-C)_n, L)|_{x=0}$. Then we have  
		\begin{equation*}
		\Delta=(C\cdot H)^2-(C\cdot\omega)^2,
		\end{equation*}
	        Since $C^2\geq 0$, we have $\Delta\leq 0$. This implies that $W(L(-C)_n, L)$ is an ellipsoid (type 1) or a paraboloid tangent to $y=l_\omega$ on the right (which cannot happen).
		Since $y_0<\mu_\omega(L(-C)_n)$ and $(x_0, y_0, z_0)$ is in the interior of $W(L(-C)_n, L)$, $W(L(-C)_n, L)$ intersects the plane $y=\mu_\omega(L(-C)_n)$, contradicting that $L(-C)_n$ is $\mu$-semistable. 
	\end{proof}
	
Now we consider possibly destabilizing objects of higher rank.  Let $E$ be a suboject of $L$ in $Coh^0$, and $K$ be the kernel of the map $f:E\to L$ in $Coh$. 	
Let $C_E$ be the curve associated to $E/K$ as in the rank $1$ case. If the context is clear, we will write $C$ for $C_E$.

If $C\neq 0$. This gives an exact sequence:
\begin{equation*}
0\to E/K\to L(-C)\to Q\to 0,
\end{equation*} 
and the map $E/K\to L$ factors through $L(-C)$. 

\begin{lem} [Arcara-Miles, \cite{AM16}]
	\label{actualwall}
	Let $\sigma_{x_0, y_0, z_0}$ be a stability condition. Let $E$ be an object which destabilizes $L$ at $\sigma_{x_0, y_0, z_0}$. Assuming that for any $F\subset L$ in $Coh^{y_0}$ with $rk(F)<rk(E)$, $W(F, L)_{z_0}\subset W(E, L)_{z_0}^<$. Then $W(E, L)_z\subset W(L(-C), L)_z^<$ for $|z|\gg 0$. 	
\end{lem}
\begin{proof}
	By Proposition~\ref{type123} we know that $W(E, L)$ is of type 4. WLOG we assume that the axis of $W(E, L)$ is positive. Since $\mu_\omega(HN_i(K))<y_0$, we have
	\begin{equation*}
	W(E, L)_z\cap \{y=\mu_\omega(HN_i(K))\}\neq\emptyset
	\end{equation*}
	for some $z$. Denote the range of such $z$ by $[a_i, b_i]$. By a similar argument of Proposition \ref{ABCH}, we see that the wall $W(E/K_1, L)_z\subset W(E, L)_z^>$ for $z\in [a_1, b_1]$. By the assumption, we have $W(E/K_1, L)_{z_0}\subset W(E, L)_{z_0}^<$. Since $W(E, L)$ and $W(E/K_1, L)$ only intersect once, the wall $W(E/K_1, L)$ can only be of type 4. Hence $W(E/K_1, L)_z\subset W(E, L)_z^>$ for all $z<b_1$. The same argument applies to $E/K_i$ and $E/K_{i+1}$ for all $i$. Finally we have $W(E/K, L)_z$ is outside $W(E, L)_z$ for all $z<b_n$. Note that $W(E/K, L)_z\subset W(L(-C), L)_z^<$ when $y<L(-C)_\omega$. Then $W(L(-C), L)_z$ is outside $W(E, L)_z$ for all $z\ll 0$. 
\end{proof}
Given a stability condition $\sigma_{x_0, y_0, z_0}$, we denote the set of sheaves satisfies the condition in Lemma~\ref{actualwall} by $U^L_{x_0, y_0, z_0}$. 
\begin{cor}
	\label{corC0}
	Let $\sigma_{x_0, y_0, z_0}$ be a stability condition such that $E\subset L$ in $Coh^{y_0}$. If $C=0$ or $C^2\geq 0$ then $E\notin U^L_{x_0, y_0, z_0}$.
\end{cor}
The following Proposition is a weaker version of Proposition 5.10 and Proposition 5.12 in \cite{AM16}, and the proof uses many ingredients from \cite{AM16}. 
\begin{prop}
	\label{type4}
If $E\in U^L_{1, 0, 0}$, then $E$ is of rank $1$, i.e. $E=L(-C_E)_n$.
\end{prop}
\begin{proof}
	 We argue by contradiction, assume that $rk(E)\geq 2$. We omit the subscript $E$ in $C_E$.
	 First consider if $C\cdot H> 0$. 
	 Based on the same analysis of the geometry of $W(E, L)$'s, we have that the wall $W(L(-C), L)$ is also tangent to the plane $y=L(-C)_\omega$ at the point 
	 \begin{equation*}
	 (0, L(-C)_\omega, L(-C)_H)=(0, l_\omega-C\cdot\omega, l_H+C\cdot H).
	 \end{equation*} 
	 By Lemma~\ref{actualwall}, $W(E, L)\subset W(L(-C), L)^<$ for $y\ll0$. We have $W(L(-C), L)$ is also of type 4, and the slope of the axes of $W(E, L)$ and $W(L(-C), L)$ are both negative. Denoting the plane where $W(L(-C), L)$ and $W(E, L)$ intersect by $z=z_0$, where $z_0>0$.
	 
	  Let $z=z_E$ be the horizontal tangent plane of the left branch of $W(E, L)$, and let $z=z_C$ be the horizontal tangent plane of the left branch of $W(L(-C), L)$. Since $W(E, L)$ and $W(L(-C), L)$ already intersect at $z=z_0$, we have $z_C>z_E$. Then the layout of $W(E, L(-C))_z$ is as follows:\\
	 \begin{itemize}
	 \item[(i)] When $z_C<z<z_0$, $W(E, L(-C))_z\subset W(E, L)_z^>$ when $y<\mu_\omega(E)$, $W(E, L(-C))_z\subset W(E, L)_z^<\cap W(L(-C), L)_z^>$ when $y>\mu_\omega(E)$;
	 \item[(ii)] When $z>z_0$, $W(E, L(-C))_z\subset W(E, L)_z^<$ when $y<\mu_\omega(E)$, and $W(E, L(-C))_z\subset W(E, L)_z^>\cap W(L(-C), L)_z^<$ when $y>\mu_\omega(E)$.
	 \end{itemize}
	 Assuming that $L(-C)_H=l_H+C\cdot H\leq z_0$. If 
	 \begin{equation*}
	 W(E, L)_{L(-C)_H}^<\cap \{y<\mu_\omega(HN_n(E))\}\neq \emptyset,
	 \end{equation*}
	 there exist points $(x, y, L(-C)_H)$ in the region such that
	 \begin{equation*}
	 \rho_{x, y, L(-C)_H}(E)>\rho_{x, y, L(-C)_H}(L(-C)),
	 \end{equation*} 
	 which contradicts the fact that $L(-C)$ is $\sigma_{x, y, L(-C)_H}$ stable. 
	 If not, then $L(-C)_H<0$ and $W(E, L)_z$ intersects $y=\mu_\omega(HN_n(E))$ at some $z$ in the range $L(-C)_H<z<0$. Inside this range, we have $W(E_{n-1}, L)_z\subset W(E, L)_z^>$. By the assumption, we have $W(E_{n-1}, L)_0\subset W(E, L)_0^<$. Hence $W(E_{n-1}, L)_{L(-C)_H}\subset W(E, L)_{L(-C)_H}^>$.
	 If 
	 \begin{equation*}
	 W(E_{n-1}, L)_{L(-C)_H}^<\cap\{y<\mu_\omega(HN_{n-1}(E))\}\neq\emptyset,
	 \end{equation*} 
	 then $E_{n-1}$ destabilizes $L(-C)$ at some $\sigma_{x, y, L(-C)_H}$ which leads to a contradiction as before. If not, we continue to consider $E_{n-2}$. Assuming that $W(E_i, L)_{L(-C)_H}^<\cap\{y<\mu_\omega(HN_i(E))\}=\emptyset$ for all $i>2$. Then if $\mu_\omega(E_1)<l_\omega$, $W(E_1, L)\cap \{y=\mu_\omega(E_1)\}\neq \emptyset$. This contradicts that $E_1$ is $\mu$-semistable. If $\mu_\omega(E_1)=l_\omega$, then $\rho_{x, y, z}(E_1)\leq \rho_{x, y, z}(L)$ for any $y<\mu_\omega(E_1)$, which contradict with $\rho_{x, y, z}(E_1)>\rho_{x, y, z}(L)$ for some $(x, y, z)\in W(E_1, L)^<$. 
	 
	 Now consider if $L(-C)_H=l_H+C\cdot H\geq z_0$. Note that $W(E(C), L)$ is obtained by translating $W(E, L(-C))$. We have $E(C)$ destabilizes $L$ at some $\sigma_{x, y, -C\cdot H}\in W(E(C), L)_{-C\cdot H}^<$. By Corollary \ref{corC0}, we have $E(C)\notin U^L_{x, y, -C\cdot H}$. Take $F\in U^L_{x, y, -C\cdot H}$. Then the center of $W(L(-C_F), L)$ is negative as $z\ll 0$. This implies that $C_F\cdot H<0$, hence $C_F^2\geq 0$ by Lemma~\ref{curve}. This contradicts with Corollary~\ref{corC0}.

	 We are left with considering when $C\cdot H<0$. By Lemma~\ref{curve} below, we have $C^2\geq 0$. This again contradicts with Corollary~\ref{corC0}. 
\end{proof}

\begin{lem} (Arcara-Miles, \cite{AM16})
	\label{curve}
	Let $C$ be an effective curve on $X$, if $C\cdot H<0$ then $C^2\geq 0$. 
\end{lem}
Finally we prove the main theorem.
\begin{thm}
	If $L$ satisfies the algebraic obstruction in Lemma \ref{lem: conjDim2}, then $L$ is $\sigma_{1, 0, 0}$-stable. 
\end{thm}
\begin{proof}
	By Proposition~\ref{type4}, we only need to show the obstruction implies that $\rho_{1, 0, 0}(L(-C))<\rho_{1, 0, 0}(L)$ for any curve $C$ of negative self-intersection.
	We write the condition in Lemma \ref{lem: conjDim2} in a more explicit form. To better compare with Bridgeland stability condition, we require that $Z_{1, 0, 0}(L)$ lies on the upper half plane removing the positive real axis.  
	Then the algebraic obstruction in Lemma \ref{lem: conjDim2} is equivalent to 
	\begin{equation}
	\label{ObsC}
	C\cdot ch_1(L)>\frac{ch_2(L)-\frac{1}{2}}{ch_1(L)\cdot\omega}(C\cdot \omega)
	\end{equation}
	for all $C\subset X$. 	
	On the other hand, $\rho_{1, 0, 0}(L(-C))<\rho_{1, 0, 0}(L)$ if and only if 
	\begin{equation*}
	\frac{ch_2(L)-C\cdot ch_1(L)+\frac{1}{2}C^2-\frac{1}{2}}{ch_1(L)\cdot\omega-C\cdot \omega}<\frac{ch_2(L)-\frac{1}{2}}{ch_1(L)\cdot\omega}.
	\end{equation*}
	This is equivalent to 
	\begin{equation*}
	C\cdot ch_1(L)-\frac{1}{2}C^2>\frac{ch_2(L)-\frac{1}{2}}{ch_1(L)\cdot\omega}(C\cdot \omega)
	\end{equation*}
	Then equation~\eqref{ObsC} implies that $L$ is $\sigma_{1, 0, 0}$-stable.		
\end{proof}

Finally, we give an example of a line bundle $L$ which is $\sigma_{1, 0, 0}$-stable, but does not satisfy the obstruction in Lemma \ref{lem: conjDim2}. Denoting the hyperplane class on $\mathbb{P}^2$ by $G$, and the exceptional divisor by $E$. We pick 
\begin{equation*}
\omega=\frac{1}{\sqrt{3}}(2G-E)
\end{equation*} and 
\begin{equation*}
H=\frac{1}{\sqrt{3}}(G-2E).
\end{equation*}
Consider the line bundle
\begin{equation*}
L=\frac{4}{\sqrt{3}}\omega-\frac{2}{\sqrt{3}}H.
\end{equation*}
Then by explicit computation, we see that $L$ satisfies 
\begin{equation*}
\frac{C\cdot ch_1(L)-\frac{1}{2}C^2}{C\cdot \omega}>\frac{ch_2(L)-\frac{1}{2}}{ch_1(L)\cdot\omega},
\end{equation*}
for $C=E$, then $L$ is $\sigma_{1, 0, 0}$ stable by Proposition 5.12 in \cite{AM16}.
However it does not satisfy the obstruction in equation~\eqref{ObsC}.

 \section{Analytic aspects and heat flows}\label{sec: AnAsp}
 
 In this final section we wish to focus on some of the analytic aspects of the dHYM equation, including some approaches using geometric flows. In fact, the first approach to understanding the existence of solutions to the dHYM equation was due to Jacob-Yau \cite{JY} using a heat flow approach.  Suppose $[\alpha] \in H^{1,1}(X,\mathbb{R})$ is a class admitting an almost calibrated representative $\alpha$.  Let $\hat{\theta}$ denote the analytic lifted angle, and consider the heat flow of functions $\phi \in C^{\infty}(X,\mathbb{R})$ given by
 \begin{equation}\label{eq: lbmcf}
 \ddt \phi= \Theta(\alpha_{\phi}) - \hat{\theta}, \qquad  \phi(0)= \phi_0
 \end{equation}
This is the {\em line bundle mean curvature flow} (LBMCF) introduced by Jacob-Yau \cite{JY}, who proved
\begin{thm}[Jacob-Yau, \cite{JY}]
Suppose $[\alpha]$ has hypercritical phase in the sense of Definition~\ref{defn: liftedAngle}, and that $(X,\omega)$ has non-negative orthogonal bisectional curvature.  Then the line bundle mean curvature flow~\eqref{eq: lbmcf} starting from initial data $\phi_0$ satisfying $\Theta(\alpha_{\phi_0}) > (n-1)\frac{\pi}{2}$ exists for all time and converges to a solution of the dHYM equation.
\end{thm}

The assumption on the orthogonal bisectional curvature, and the initial data is rather restrictive. On the other hand, as pointed out in \cite{JY}, if $\alpha$ is K\"ahler then $k\alpha$ will satisfy the condition $\Theta(\alpha_{\phi_0}) > (n-1)\frac{\pi}{2}$ for some $k\gg 1$.  In order to remove, or at least weaken, these assumptions the first author, with Jacob and Yau considered the problem from the elliptic point of view \cite{CJY15}.  Building on ideas of Sz\'ekelyhidi \cite{Sz} and Wang-Yuan \cite{WY} the authors proved
\begin{thm}[C.-Jacob-Yau, \cite{CJY15}]\label{thm: existence thm}
Let $(X,\omega)$ be a compact K\"ahler manifold with complex dimension $n$, and fix a cohomology class $[\alpha]\in H^{1,1}(X,\mathbb{R})$.  Suppose that 
\[
\int_{X}(\omega+\sqrt{-1}\alpha)^n \in \mathbb{R}_{>0}e^{i\hat{\theta}}
\]
for $\hat{\theta} \in ((n-2)\frac{\pi}{2}, n\frac{\pi}{2})$.  Suppose that there exists a form $\chi:=\alpha+\ddb\underline{\phi} \in [\alpha]$ with the following property: at each point $x \in X$, if $\mu_1, \dots, \mu_n$ denote the eigenvalues of $\alpha+\ddb \underline{\phi}$ with respect to $\omega$, then, for all $j=1, \dots,n$ we have
\begin{equation}
\label{eq: Csub}
\sum_{\ell \ne j}  \arctan(\mu_\ell) > \hat{\theta} -\frac{\pi}{2}.
\end{equation}
Furthermore, assume
\begin{equation}\label{eq: extraAss}\Theta(\chi) >(n-2)\frac{\pi}{2}.
\end{equation}
Then there exists a smooth function $\phi: X\rightarrow \mathbb{R}$, unique up to addition of a constant, such that $\alpha_{\phi}= \alpha+\ddb \phi $ solves the deformed Hermitian-Yang-Mills equation.
\end{thm}

In fact, it is not hard to see that the existence of $\underline{\phi}$ in Theorem~\ref{thm: existence thm} is both necessary and sufficient. The authors also noted in \cite{CJY15} that the inequalities in Corollary~\ref{cor: obstructions} (ii) were necessary for the existence of such a function $\underline{\phi}$.  Note that the condition~\eqref{eq: extraAss} is automatically satisfied if $\hat{\theta} > (n-2 + \frac{2}{n}) \frac{\pi}{2}$.   However, in the range $\hat{\theta} \in ((n-2)\frac{\pi}{2}, (n-2 + \frac{2}{n}) \frac{\pi}{2})$ it is expected that assumption~\eqref{eq: extraAss} can be dropped.  This has been confirmed in dimension $2$ \cite{JY}, and dimension $3$ \cite{Ping}.

Several authors have recently studied the LBMCF.  In \cite{HJ1} Han-Jin prove a local stability theorem for the flow which does not assume anything about $\hat{\theta}$.  Namely, they prove the convergence of the LBMCF for any $\hat{\theta}$, provided the flow starts from initial data which is $C^2$-close to a solution of the dHYM equation.  Han-Yamamoto \cite{HY} proved an $\epsilon$-regularity theorem for the flow, using a Huisken-type monotonicity formula.  In the case of K\"ahler surfaces, Takahashi \cite{Tak1} studied the LBMCF for unstable classes.  The main result of \cite{Tak} is that the LBMCF starting from sufficiently positive initial data converges to a solution of the dHYM equation away from a finite union of curves with self intersection $-1$; see \cite{SW, FLSW}.  In general the LBMCF seems more difficult to handle than the elliptic approach of \cite{CJY15}.  For example, to the authors' knowledge, there is no parabolic proof of Theorem~\ref{thm: existence thm} when $\hat{\theta}\in ((n-2)\frac{\pi}{2}, (n-1)\frac{\pi}{2})$; two useful references in this direction are \cite{PT, CPW}.

Very recently, motivated by the GIT approach of the first author and Yau to the dHYM equation, Takahashi \cite{Tak1} introduced the flow
\begin{equation}\label{eq: tanFlow}
\ddt \phi= \tan(\Theta(\alpha_{\phi}) - \hat{\theta}), \qquad  \phi(0)= \phi_0.
\end{equation}
This flow is precisely the gradient flow of the Kempf-Ness functional $\mathcal{J}$ introduced in Section~\ref{sec: GIT}.  According to finite dimensional GIT \cite{GIT, ThNo}, this flow (which is conjugate to the gradient flow of the norm squared of the moment map) should converge to a zero of the moment map, when it exists, or to the `optimal destabilizer' when no zero of the moment map exists.  For this reason it would be very interesting to understand the behavior of ~\eqref{eq: tanFlow} in examples where no solution of the dHYM equation exists.  

These ideas have been put to good use in other problems in complex geometry relating algebraic geometry and nonlinear PDE.  For example, in the study of the Hermitian-Yang-Mills equation on  holomorphic vector bundles, the gradient flow of the Kempf-Ness functional is the Donaldson heat flow, which converges (modulo gauge) to the double dual of the associated graded of the Harder-Narasimhan-Seshadri filtration \cite{Do3, UY, BS, DW}. Recently, Haiden-Katzarkov-Kontsevich-Pandit \cite{HKKP, HKKP1} have studied the Donaldson heat flow on Riemann surfaces \cite{HKKP1}, and a discrete version of the Donaldson heat flow on quivers, relating the limiting behavior to refinements of Harder-Narasimhan filtrations in Bridgeland stability conditions. In the study of K\"ahler-Einstein metrics on Fano varieties, the first author with Hisamoto and Takahashi \cite{CHT} introduced the inverse Monge-Amp\`ere flow, which is the gradient flow of the Kempf-Ness functional in Donaldson's GIT approach to K\"ahler-Einstein metrics \cite{Do4}. In \cite{CHT} the asymptotics of the inverse Monge-Amp\`ere flow were related to optimal degenerations for toric Fano varieties.  These ideas have been put to use to study the existence of Harder-Narasimhan type filtrations for general Fano manifolds \cite{His, Xia}.  

In regards to the flow~\eqref{eq: tanFlow}, the main result of \cite{Tak1} is the following

\begin{thm}[Takahashi, \cite{Tak1}]\label{thm: convTanFlow}
Let $(X,\omega)$ be compact K\"ahler, and suppose $[\alpha] \in H^{1,1}(X,\mathbb{R})$ has hypercritical phase $\hat{\theta} \in ((n-1)\frac{\pi}{2}, n\frac{\pi}{2})$ in the sense of Definition~\ref{defn: liftedAngle}.  If $\phi_0 \in \mathcal{H}$ is any almost calibrated initial data, then the flow~\eqref{eq: tanFlow} starting from $\phi_0$ exists for all time.  Furthermore, if there is a $(1,1)$ form $\chi \in [\alpha]$ satisfying~\eqref{eq: Csub}, then the flow converges smoothly to a solution of the dHYM equation.
\end{thm}

A similar result for the LBMCF can be established following the proof of Theorem~\ref{thm: existence thm}, as noted in \cite[Remark 7.4]{CJY15}.  The main new technical ingredient in the proof of Theorem~\ref{thm: convTanFlow} is a concavity result for the operator $\phi \mapsto \tan(\Theta(\alpha_{\phi}) -\hat{\theta})$.  To describe this result, and put it in context, we make a brief digression to describe the Dirichlet problem for the Lagrangian phase operator.

Let $\mathcal{S}$ denote the set of real or hermitian $n\times n$ matrices.  For $A\in \mathcal{S}$ define
\begin{equation}\label{eq: phaseOP}
\Theta(A) = \sum_{i=1}^{n}\arctan(\lambda_i)
\end{equation}
where $\lambda_i, 1 \leq i \leq n$ are the eigenvalues of $A$.  It is easy to see that the operator $\Theta(A)$ is elliptic, in the sense that if $A, B\in \mathcal{S}$ and $B-A$ is non-negative definite, then $\Theta(B) \geq \Theta(A)$.  Then then makes sense to study the following local PDE question, which was first posed by Harvey-Lawson \cite{HL1}

\begin{que}\label{que: DP}
Suppose $\Omega \subset \mathbb{R}^n$ (resp. $\mathbb{C}^n$). Fix continuous functions $\phi : \del \Omega \rightarrow \mathbb{R}$ and $f: \overline{\Omega} \rightarrow (-n\frac{\pi}{2}, n\frac{\pi}{2})$.  Then does there exist a (viscosity) solution $u:\overline{\Omega}\rightarrow \mathbb{R}$ of the Dirichlet problem
\begin{equation}\label{eq: DP}
\Theta( D^{2}u) =f \quad (\,\text{ {\rm resp.}}\,\, \Theta(\del\dbar u) =f\, )\qquad u|_{\del \Omega} = \phi.
\end{equation}
What is the regularity of $u$?
\end{que}

We have been somewhat vague in formulating Question~\ref{que: DP} in order to allow maximum flexibility.  Question~\ref{que: DP} has inspired a tremendous amount of interest and work over the past $40$ years; a thorough treatment of the state of knowledge concerning Question~\ref{que: DP} would require an entire article on its own. The first significant result related to Question~\ref{que: DP} was obtained by Caffarelli-Nirenberg-Spruck \cite{CNS3} who proved the existence of solutions for $f(x)=\pm (n-1)\frac{\pi}{2}$ when $n$ is odd and $h(x)= \pm (n-2)\frac{\pi}{2}$ when $n$ is even, under an assumption on the geometry of $\del \Omega$.  Viscosity solutions to \eqref{eq: DP} for $f =\hat{\theta} \in (-\frac{n\pi}{2}, n\frac{\pi}{2})$ constant were obtained by Harvey-Lawson in \cite{HL} for domains $\Omega$ satisfying certain geometric conditions on their boundary. 

It turns out the Harvey-Lawson viscosity solutions to~\eqref{eq: DP} with constant right-hand side are not always classical solutions (that is, they may fail to be $C^2$).  Examples of this phenomenon were first constructed by Nadirashvili-Vl\u{a}du\c{t} \cite{NV}; see also \cite{WY1}.  In order to obtain regularity, one needs to exploit some additional structure of the operator $\Theta$, such as convexity or concavity.  In this direction Yuan \cite{Y, WY} showed that the map $A \mapsto \Theta(A)$ is concave on the set $\{A \in \mathcal{S}: \Theta(A) \geq (n-1)\frac{\pi}{2}\}$, and that the level sets $\{A \in \mathcal{S} : \Theta(A) =c\}$ are convex for $c\in ((n-2)\frac{\pi}{2}, n\frac{\pi}{2})$. Furthermore, this is sharp in the sense that for $|c| < (n-2)\frac{\pi}{2}$ the level sets of $\Theta$ are not convex.  The first author, Picard and Wu \cite{CPW} strengthened Yuan's result by showing that for every $\delta >0$ there is a constant $C = C(\delta)$ such that $A\mapsto e^{-C \Theta(A)}$ is concave on the set $\{A\in \mathcal{S}: \Theta(A) > (n-2)\frac{\pi}{2} +\delta\}$.  Using this result, \cite{CPW} obtained the existence of classical solutions to~\eqref{eq: DP} for $f, \phi$ sufficiently regular, provided $f(\overline{\Omega})\subset ((n-2)\frac{\pi}{2}, n\frac{\pi}{2})$ and $\Omega$ admits a subsolution; for example, strictly convex domains, or domains satisfying the geometric assumptions of \cite{CNS3} are covered by this result.  In the same setting, Dinew-Do-T\^{o} \cite{DDT} proved the existence of viscosity solutions assuming only $f, \phi$ continuous.

In this vein, Takahashi's result \cite{Tak1} obtains a new concavity result for~\eqref{eq: phaseOP}; he shows that for $\hat{\theta} \in ((n-1)\frac{\pi}{2}, n\frac{\pi}{2})$, the map $A\mapsto \tan(\Theta(A)- \hat{\theta})$ is concave on the set $\{A\in \mathcal{S} : |\Theta(A)-\hat{\theta}|<\frac{\pi}{2}\}$.  Harvey-Lawson  \cite{HL3} proved a related ``tameness" result for the operator $A\mapsto \tan(\frac{\Theta(A)}{n})$ on the set $\{A \in \mathcal{S} : \Theta(A) > (n-2)\frac{\pi}{2} +\delta\}$ for $\delta >0$.  Based on their general theory for tamable elliptic equations \cite{HL2}, Harvey-Lawson provided another proof of the result of Dinew-Do-T\^{o}.  It would be very interesting to understand if Takahashi's flow~\eqref{eq: tanFlow} can shed any light on the issue of Harder-Narasimhan filtrations for unstable line bundles.


\begin{thebibliography}{99}

\bibitem{AB} D. Arcara, and A. Bertram {\em Bridgeland-stable moduli spaces for $K$-trivial surfaces}, J. Eur. Math. Soc. (JEMS) {\bf 15} (2013), no. 1, 1--38.

\bibitem{ABCH13}
D. Arcara, A. Bertram, I. Coskun and J. Huizenga, {\em The minimal model program for the Hilbert scheme of points on $\mathbb{P}^2$ and Bridgeland stability}, Adv. Math., {\bf235} (2013), 580-626.

\bibitem{AM16}
D. Arcara and E. Miles, {\em Bridgeland stability of line bundles on surfaces}, J. Pure App. Alg., {\bf 220}, (2016), no. 4, 1655-1677

\bibitem{AtBo} M. F. Atiyah, and R. Bott {\em The Yang-Mills equations over Riemann surfaces}, Philos. Trans. Roy. Soc. London Ser. A {\bf 208} (1983), no. 1505, 523--615.

\bibitem{BS} S. Bando, and Y.-T. Siu {\em Stable sheaves and Einstein-Hermitian metrics}, Geometry and Analysis on Complex Manifolds, World Sci. Publ., River Edge, NJ, 1994

\bibitem{BM11} A. Bayer, and E. Macri {\em The space of stability conditions on the local projective plane}, Duke Math. J., {\bf 160} (2011) no. 2, 263-322.

\bibitem{BMT} A. Bayer, E. Macri, and Y. Toda {\em Bridgeland stability conditions on threefolds I: Bogomolov-Gieseker type inequalities}, J. Algebraic Geom.  {\bf 23} (2014), 117--163.

\bibitem{BBS} K. Becker, M. Becker, and A. Strominger {\em Fivebranes, membranes and non-perturbative string theory}, Nuclear Phys. B {\bf 456} (1995), no. 1-2, 130--152

\bibitem{Br} T. Bridgeland {\em Stability conditions on triangulated categories}, Ann. of Math. (2) {\bf 166} (2007), no. 2, 317--345.

\bibitem{Bri08} T. Bridgeland {\em Stability conditions on K3 surfaces}, Duke Math. J. {\bf 141} (2008), no. 2, 241-291.

\bibitem {CNS3} Caffarelli, L.A., Nirenberg, L., and Spruck, J., {\em The Dirichlet problem for nonlinear second order elliptic equations, III: Functions of the eigenvalues of the Hessian}, Acta Math. 155 (1985), 261-301.

\bibitem{Cal} E. Calabi, {\em Extremal K\"ahler metrics}, Seminar on Differential Geometry, Ann. of Math. Stud., 102, Princeton Univ. Press, Princeton, N.J., 1982.

\bibitem{GC} G. Chen, {\em On $J$-equation}, preprint, arXiv:1905.10222

\bibitem{CCL} J. Chu, T. C. Collins, and M.-C. Lee, {\em The space of almost calibrated $(1,1)$ forms on a compact K\"ahler manifold}, arXiv:2002.01922

\bibitem{CHT} T. C. Collins, T. Hisamoto, and R. Takahashi {\em The inverse Monge-Amp\`ere flow and applications to K\"ahler-Einstein metrics}, J. Differential Geom., to appear

\bibitem{CJY15}  T. C. Collins, A. Jacob, and S.-T. Yau, {\em $(1,1)$ forms with specified Lagrangian phase: A priori estimates and algebraic obstructions}, Cambr. J. Math., to appear

\bibitem{CPW} T. C. Collins, S. Picard, and X. Wu {\em Concavity of the Lagrangian phase operator and applications}, Calc. Var. Partial Differential Equations {|bf 56} (2017), no. 4, Art. 89.

\bibitem{CoSz} T. C. Collins, and G. Sz\'ekelyhidi {\em Convergence of the $J$-flow on toric manifolds}, J. Differential Geom. {\bf 107} (2017), no. 1, 47--81.

\bibitem{CXY} T. C. Collins, D. Xie, and S.-T. Yau {\em The deformed Hermitian-Yang-Mills equation in geometry and physics}, Geometry and Physics, Volume I: A Festschrift in honour of Nigel Hitch, {\em Oxford University Press}, December, 2018.

\bibitem{CY18} T. C. Collins, and S.-T. Yau, {\em Moment maps, nonlinear PDE, and stability in mirror symmetry}, preprint, arXiv:1811.04824.


\bibitem{DarRu} T. Darvas, and Y. Rubinstein {\em A minimum principle for Lagrangian graphs}, Comm. Anal. Geom. {\bf 27} (2019), no. 4, 857--876

\bibitem{DL12} T. Darvas, and L. Lempert, {\em Weak geodesics in the space of K\"ahler metrics}, Math. Res. Lett. {\bf 19} (2012), no. 5, 1127--1135.

\bibitem{DW} G. Daskalopoulos, and R. Wentworth {\em Convergence properties of the Yang-Mills flow on K\"ahler surfaces}, J. Reine Angew. Math. {\bf 575} (2004), 69--99

\bibitem{Del} M. Dellatorre {\em The degenerate special Lagrangian equation on Riemannian manifolds}, preprint, arXiv:1709.00496

\bibitem{DP} J.-P. Demailly, and M. P\u{a}n {\em Numerical characterization of the K\"ahler cone of a compact K\"ahler manifold}, Ann. of Math. (2) {\bf 159}(2004), no. 3, 1247--1274.

\bibitem{DDT} S. Dinew, H.-S. Do, and T. D. T\^{o} {\em A viscosity approach to the Dirichlet problem for degenerate complex Hessian-type equations}, Anal. ODE {\bf 12} (2019), no. 2, 505--535.

\bibitem{Do} S.K. Donaldson, {\em Moment maps in differential geometry}, Surveys in differential geometry, Vol. VIII (Boston, MA, 2002), 171--189, Surv. Differ. Geom., {\bf 8}, Int. Press, Somerville, MA, 2003.

\bibitem{Do1} S. K. Donaldson {\em Symmetric spaces, K\"ahler geometry and Hamiltonian dynamics}, Northern California Symplectic Geometry Seminar, 13--33, Amer. Math. Soc. Transl. Ser. 2, {\bf 196}, Adv. Math. Sci., 45, Amer. Math. Soc., Providence, RI, 1999.

\bibitem{Do3} S. K. Donaldson {\em Anti self-dual Yang-Mills connections over complex algebraic surfaces and stable vector bundles}, Proc. London Math. Soc. (3) {\bf 50} (1985), 1--26.

\bibitem{Do4} S. K. Donaldson {\em The Ding functional, Berndtsson convexity, and moment maps}, Geometry, analysis and probability, 57--67, Progr. Math., 310, Birkh\"auser/Springer, Cham, 2017

\bibitem{Do5} S. K. Donaldson {\em Lower bounds on the Calabi functional}, J. Differential Geom. {\bf 70} (2005), no. 3, 453--472.


\bibitem{Doug} M. Douglas {\em Dirichlet branes, homological mirror symmetry, and stability}, Proceedings of the International Congress of Mathematicians, Vol. III (Beijing, 2002), 395--408, Higher Ed. Press, Beijing, 2002.

\bibitem{DFR} M. Douglas, B. Fiol, and C. R\"omelsburger {\em Stability and BPS branes}, J. High Energy Phys. (2005), no. 9

\bibitem{FLSW} H. Fang, M. Lai, J. Song, and B. Weinkove {\em The $J$-flow on K\"ahler surfaces: a boundary case}, Anal. PDE {\bf 7} (2014), 215--226.

\bibitem{HKKP} F. Haiden, L. Katzarkov, M. Kontsevich, and P. Pandit {\em Semistability, modular lattices, and iterated logarithms} arXiv:1706.01073

\bibitem{HKKP1} F. Haiden, L. Katzarkov, M. Kontsevich, and P. Pandit {\em Iterated logarithms and gradient flows}, arXiv:1802.04123

\bibitem{HJ1} X. Han and X. Jin {\em Stability of line bundle mean curvature flow}, preprint, arXiv:2001.07406

\bibitem{HY} X. Han, and H. Yamamoto {\em An $\epsilon$-regularity theorem for line bundle mean curvature flow}, preprint, arXiv:1904.02391

\bibitem{HL1} F. R. Harvey, and H. B. Lawson Jr. {\em Calibrated geometries}, Acta Math. {\bf 148} (1982), 47--157

\bibitem{HL} F. R. Harvey, and H. B. Lawson Jr. {\em Dirichlet duality and the nonlinear Dirichlet problem}, Comm. Pure Appl. Math. {\bf 62} (2009), 396--443.

\bibitem{HL3} F. R. Harvey, and H. B. Lawson Jr. {\em Pseudoconvexity for the special Lagrangian potential equation}, preprint, arXiv:2001.09818

\bibitem{HL2} F. R. Harvey, and H. B. Lawson Jr. {\em The inhomogeneous Dirichlet problem for natural operators on manifolds}, preprint, arXiv:1805.11121

\bibitem{His} T. Hisamoto, {\em Geometric flow, multiplier ideal sheaves and optimal destabilizer for a Fano manifold}, preprint, arXiv:1901.08480

\bibitem{Hit} N. Hitchin, {\em The moduli space of special Lagrangian submanifolds}, Ann. Scuola Norm. Sup. Pisa Cl. Sci. (4) {\bf 25} (1997), no. 3--4

\bibitem{Hit1} N. Hitchin, {\em The self-duality equations on a Riemann surface}, Proc. London Math. Soc. (3) {\bf 1987}, no. 1, 59--126.

\bibitem{Jac2} A. Jacob {\em The Yang-Mills flow and the Atiyah-Bott formula on compact K\"ahler manifolds}, Amer. J. Math. {\bf 138} (2016), no. 2, 329--365

\bibitem{Jac} A. Jacob, {\em Weak geodesics for the deformed Hermitian-Yang-Mills equation}, Pure Appl. Math. Q., to appear

\bibitem{JS} A. Jacob, and N. Sheu, {\em The deformed Hermitian-Yang-Mills equation on the blowup of $\mathbb{P}^n$}, in preparation

\bibitem{JY} A. Jacob, and S.-T. Yau {\em A special Lagrangian type equation for holomorphic line bundles}, Math. Ann. {\bf 369} (2017), no. 1-2, 869--898.

\bibitem{J} D. Joyce {\em Conjectures on Bridgeland stability for Fukaya categories of Calabi-Yau manifolds, special Lagrangians, and Lagrangian mean curvature flow}, EMS Surv. Math. Sci. {\bf 2} (2015), no. 1, 1--62

\bibitem{Kis} C. O. Kiselman {\em The partial Legendre transform for plurisubharmonic functions}, Invent. Math. {\bf 49} (1978), no. 2, 137--148

\bibitem{Kont} M. Kontsevich {\em Homological algebra of mirror symmetry}, Proceedings of the International Congress of Mathematicians, Vol. 1, 2 (Z\"urich, 1994), 120--139, Birkh\"auser, Basel, 1995

\bibitem{LV13} L. Lempert, and L. Vivas, {\em Geodesics in the space of K\"{a}hler metrics}, Duke Math. J. {\bf 162} (2013), no. 7, 1369--1381.

\bibitem{LYZ} N. C. Leung, S.-T. Yau, and E. Zaslow {\em From special Lagrangian to Hermitian-Yang-Mills via Fourier-Mukai}, Adv. Theor. Math. Phys. {\bf 4} (2000), no. 6, 1319--1341

\bibitem{Li} C. Li {\em On stability conditions for the quintic threefold}, Invent. Math {\bf 218} (2019), no. 1, 301--340

\bibitem{Mac14}
A. Maciocia, {\em Computing the walls associated to Bridgeland stability conditions on projective surfaces}, Asian J. Math., {\bf 18}, (2014), no. 2, 263-280

\bibitem{MMMS} M. Mari\~no, R. Minasian, G. Moore, and A. Strominger {\em Nonlinear instantons from supersymmetric $p$-branes}, J. High Energy Phys. (2000), no. 1

\bibitem{GIT} D. Mumford, J. Fogarty, and F. Kirwan {\em Geometric invariant theory}, Third Edition, Erg. Math. {\bf 34}, Springer-Verlag, Berlin.

\bibitem{NV} Nadirashvili, N., and Vl\u{a}du\c{t}, S. {\em Singular solution to the Special Lagrangian Equations}, Ann. Inst. H. Poincar\'e Anal. Non Lin\'eaire 27 (2010), no. 5, 1179-1188

\bibitem{PT} D. H. Phong, and D. T. T\^{o} {\em Fully non-linear parabolic equations on compact Hermitian manifolds}, preprint, arXiv:1711.10697

\bibitem{Ping} V. Pingali, {\em The deformed Hermitian-Yang-Mills equation on three-folds}, arXiv:1910.01870

\bibitem{RT} J. Ross, and R. Thomas, {\em A study of the Hilbert-Mumford criterion for the stability of projective varieties}, J. Algebraic Geom. {\bf 16} (2007), no. 2, 201--255

\bibitem{RWN} J. Ross, and D. Witt Nystr\"om {\em The minimum principle for convex subequations}, preprint, arXiv:1806.06033

\bibitem{RuSol} Y. Rubinstein, and J. P. Solomon {\em The degenerate special Lagrangian equation}, Adv. Math. {\bf 310} (2017), 889--939.

\bibitem{ScSt} E. Schlitzer, and J. Stoppa {\em Deformed hermitian Yang-Mills connections, extended gauge group and scalar curvature}, preprint, arXiv: 1911.10852

\bibitem{Sch} B. Schmidt {\em Counterexample to the generalized Bogomolov-Gieseker inequality for threefolds}, Int. Math. Res. Not. IMRN 2017, no. 8, 2562--2566



\bibitem{Solomon14} J. P. Solomon, {\em Curvature of the space of positive Lagrangians}, Geom. Funct. Anal. {\bf 24} (2014), no. 2, 670--689.

\bibitem{Sol} J. P. Solomon {\em The Calabi homomorphism, Lagrangian paths and special Lagrangians}, Math. Ann {\bf 357} (2013), no. 4, 1389--1424.

\bibitem{SW} J. Song and B. Weinkove {\em On the convergence and singularities of the $J$-flow with applications to the Mabuchi energy}, Comm. Pure Appl. Math. {\bf 61} (2008), no. 2, 210-229.

\bibitem{Sz} G. Sz\'ekelyhidi {\em Fully non-linear elliptic equations on compact Hermitian manifolds}, J. Differential Geom. {\bf 109} (2018), no. 2, 337--378.


\bibitem{Th} R. P. Thomas {\em Moment maps, monodromy, and mirror manifolds}, Symplectic geometry and mirror symmetry (Seoul, 2000), 467--498, World Sci. Publ., River Edge, NJ, 2001.

\bibitem{Tak} R. Takahashi {\em Collapsing of the line bundle mean curvature flow on K\"ahler surfaces}, preprint, arXiv:1912.13145

\bibitem{Tak1} R. Takahashi {\em Tan-concavity for the Lagrangian phase operators and applications to the tangent Lagrangian phase flow}, preprint, arXiv:2002.05132v5

\bibitem{ThY} R. P. Thomas, and S.-T. Yau {\em Special Lagrangians, stable bundles, and mean curvature flow}, Comm. Anal. Geom. {\bf 10} (2002), no. 5, 1075--1113.

\bibitem{ThNo} R. P. Thomas {\em Notes on GIT and symplectic reduction for bundles and varieties}, Surv. Differ. Geom., 10, Int. Press, Somerville, MA, 2006.

\bibitem{UY} K. Uhlenbeck, and S.-T. Yau {\em On the existence of Hermitian-Yang-Mills connections in stable vector bundles}, Comm. Pure Appl. Math. {\bf 39-S} (1986), 257--293.

\bibitem {WY} Wang, D., and Yuan, Y.  {\em Hessian estimates for special Lagrangian equations with critical and supercritical phases in general dimensions}, Amer. J. Math., 136 (2014), 481-499.

\bibitem{WY1} Wang, D., and Yuan, Y. {\em Singular solutions to the special Lagrangian equations with subcritical phases and minimal surface systems}, Amer. J. Math 135 (2013), no. 5, 1157-1177.

\bibitem{Xia} M. Xia {\em On sharp lower bounds for Calabi type functionals and destabilizing properties of gradient flows}, preprint, arXiv:1901.07889

\bibitem{Y} S.-T. Yau {\em On the Ricci curvature of a compact K\"ahler manifolds and the complex Monge-Amp\`ere equation, I}, Comm. Pure. Appl. Math {\bf 31} (1978), no. 3, 339-411.

\bibitem{Yuan} Y. Yuan {\em Global solutions to special Lagrangian equations}, Proc. Amer. Math. Soc. {\bf 134} (2006), no. 5, 1355--1358.



\end{thebibliography}
\end{document}